\documentclass[10pt]{amsart}

\usepackage{amsmath,amssymb,graphicx,wasysym,tikz}
\usepackage[utf8]{inputenc}

\usetikzlibrary{arrows, positioning, shapes, decorations.pathreplacing, decorations.markings, calc, matrix}

\tikzset{->-/.style={decoration={
  markings,
  mark=at position #1 with {\arrow{>}}},postaction={decorate}}}

\usepackage[pdftex]{hyperref}

\theoremstyle{plain}

\newtheorem{theorem}{Theorem}[section]

\newtheorem{lemma}[theorem]{Lemma}
\newtheorem{proposition}[theorem]{Proposition}
\newtheorem{corollary}[theorem]{Corollary}

\theoremstyle{remark}
\newtheorem{example}[theorem]{Example}
\newtheorem{remark}[theorem]{Remark}

\theoremstyle{definition}
\newtheorem{definition}[theorem]{Definition}
\newcommand{\floor}[1]{\left\lfloor #1 \right\rfloor}

\newcommand{\BC}{\mathbf{C}}
\newcommand{\BP}{\mathbf{P}}

\newcommand{\BR}{\mathbf{R}}

\newcommand{\CG}{\mathcal{G}}

\newcommand{\CO}{\mathcal{O}}
\newcommand{\CP}{\mathcal{P}}

\newcommand{\CX}{\mathcal{X}}

\newcommand{\norm}[1]{\left| #1 \right|}
\newcommand{\normsq}[1]{\left| #1 \right|^2}

\newcommand{\suchthat}{ \ | \ }

\newcommand{\Hom}{\mathrm{Hom}}

\newcommand{\Rep}{\mathrm{Rep}}

\newcommand{\gl}{\mathfrak{gl}}

\newcommand{\fg}{\mathfrak{g}}

\newcommand{\fgd}{\mathfrak{g}^\ast}

\newcommand{\Tr}{\mathrm{Tr}}
\newcommand{\rk}{\mathrm{rk}}

\newcommand{\red}{/\!/}
\newcommand{\reda}[1]{ \underset{#1}{/\!/} }
\newcommand{\rred}{/\!/\!/}
\newcommand{\rreda}[1]{ \underset{#1}{/\!/\!/} }

\newcommand{\diag}{\mathrm{diag}}

\newcommand{\bd}{\mathbf{d}}

\newcommand{\mur}{\mu_r}
\newcommand{\muc}{\mu_c}

\newcommand{\into}{\hookrightarrow}

\newcommand{\iso}{\cong}

\newcommand{\ev}[1]{\left\langle{#1}\right\rangle}

\newcommand{\Endo}{\mbox{End}}
\providecommand{\set}[1]{\left\{ #1\right\}}

\newcommand{\Omin}{\mathcal{O}_{\mathrm{min}}}
\newcommand{\Oreg}{\mathcal{O}_{\mathrm{reg}}}
\newcommand{\Omincl}{\overline{\mathcal O}_{\mathrm{min}}}
\newcommand{\Oregcl}{\overline{\mathcal O}_{\mathrm{reg}}}
\begin{document}
\title[Hyperpolygons and Hitchin Systems]{Hyperpolygons and Hitchin Systems}
%%\date{\today}

\author{Jonathan Fisher}
\address{Fachbereich Mathematik, Universit\"at Hamburg, Germany}
\email{jonathan.fisher@uni-hamburg.de}

\author{Steven Rayan}
\address{Department of Mathematics, University of Toronto, Canada}
\email{stever@math.toronto.edu}

\begin{abstract} 
We study the hyperk\"ahler analogues of moduli spaces of semistable $n$-gons in complex projective space.  We prove that the hyperk\"ahler Kirwan map is surjective and produce a formula that recursively calculates the Betti numbers of these spaces for all ranks.  Building on a natural analogy between hyperpolygons and parabolic Higgs bundles, we identify hyperpolygon spaces with certain degenerate Hitchin systems, and use this to establish their complete integrability, for ranks up to and including $3$.
\end{abstract}



\maketitle

\tableofcontents

%==============================================================================
%==============================================================================

%------------------------------------------------------------------------------
%------------------------------------------------------------------------------

%=========================================================================================
% "Hyperpolygons and Hitchin Systems"
%
% Jonathan Fisher
% jonathan.fisher@uni-hamburg.de
%
% Steven Rayan
% stever@math.toronto.edu
% 
% TeX source for Section 1: "Introduction"
%
%=========================================================================================

%=========================================================================================
\section{Introduction}
%=========================================================================================

In this paper, we study a very basic problem in geometric invariant theory: the moduli of polygons in complex projective space.  We denote by $\CP_n^r(\alpha)$ the moduli space of $n$-gons in $\BP^{r-1}$, where $\alpha \in \BR^n$ is a K\"ahler modulus. The case $r=2$ was first studied by Klyachko \cite{Klyachko}, who gave a procedure to compute their Betti numbers by studying certain Hamiltonian flows.

For all $r$, the spaces $\CP^r_n(\alpha)$ can be constructed as compact symplectic quotients, and hence their Betti numbers and cohomology rings may be computed using the standard techniques \cite{JeffreyKirwan95,TolmanWeitsman}.

We focus on \emph{hyperpolygon} spaces $\CX^r_n(\alpha)$, which are the hyperk\"ahler analogues of polygon spaces.  For hyperpolygons, the $r=2$ case was studied by Konno \cite{KonnoPolygon}, who computed their Betti numbers and cohomology rings. Our first result is the following, in which $H^\ast$ denotes cohomology with rational coefficients, as will be the convention throughout.

%-----------------------------------------------------------------------------------------
\begin{theorem}
Let $G = S(U(r) \times (S^1)^n)$. The natural hyperk\"ahler Kirwan map
\begin{equation}
  \kappa: H^\ast(BG) \to H^\ast(\CX^r_n(\alpha))\nonumber
\end{equation} 
is surjective.
\end{theorem}
%-----------------------------------------------------------------------------------------

Our second main result is an explicit recursive procedure for computing the Betti numbers of $\CX^r_n(\alpha)$, again for all $r\geq2$.  In contrast to Konno, who computed the Betti numbers using a natural $S^1$-action on these spaces, we use equivariant Morse theory with the norm-square of a certain moment map. What we prove is\\

%-----------------------------------------------------------------------------------------
\begin{theorem}
The Poincar\'e polynomial $P_t(\CX^r_n(\alpha))$ is independent of $\alpha$ and may be computed by the recursion relation
\begin{equation}
  \frac{P_t(Gr(r,n))}{(1-t^2)^{n-1}} 
    = \sum_{\lambda} \frac{1}{m(\lambda)!} \sum_{\rho \geq \lambda} \frac{t^{2 \beta(\lambda, \rho)}}{(1-t^2)^{s(\lambda, \rho)}}
      {n \choose \rho }
      \prod_{j=1}^{\ell(\lambda)} P_t(\CX^{\lambda_j}_{\rho_j}),\nonumber\end{equation}
where the outermost sum is taken over all integer partitions $\lambda$ of $n$; $\rho$ is a tuple of non-negative integers with the same number of entries as $\lambda$; $\rho\geq\lambda$ refers to lexicographical ordering; and $m(\lambda)!$, $\beta(\lambda,\rho)$, $s(\lambda,\rho)$, and ${n \choose \rho }$ are numbers whose definitions we leave to \S\ref{sec-betti}
\end{theorem}
%-----------------------------------------------------------------------------------------

%-----------------------------------------------------------------------------------------
\begin{remark} Although the right side of this equation contains an infinite sum over $\rho$, due to the presence of the multinomial coefficient ${n \choose \rho}$ only finitely many of these terms are non-zero. 
\end{remark}
%-----------------------------------------------------------------------------------------

%-----------------------------------------------------------------------------------------
\begin{remark} Similar recursion relations for Nakajima quiver varieties can be deduced using Hausel's arithmetic Fourier transform \cite{HauselBetti}. We give an independent proof of Theorem \ref{main-thm-betti} using only standard Morse theory. It is intriguing that the same combinatorial structures appear both in types of calculations.
\end{remark}
%-----------------------------------------------------------------------------------------

The holomorphic symplectic geometry of hyperpolygon spaces, studied in the final part of this paper, is tied naturally to the theory of Higgs bundles, and hence to known integrable systems --- namely, the Hitchin system.  In the recent work of Godinho-Mandini \cite{GodinhoMandini}, it was observed that the spaces $\CX^2_n(\alpha)$ may be identified with an open subset of a moduli space of stable rank 2 parabolic Higgs bundles on $\BP^1$, and in fact this map turns out to be a symplectomorphism \cite{BFGM}. This establishes the complete integrability of the moduli space $\CX^2_n(\alpha)$ with respect to its natural symplectic structure.  We extend this construction to to all ranks. 

%-----------------------------------------------------------------------------------------
\begin{theorem}
The space $\CX^r_n(\alpha)$ may be identified with a moduli space of rank $r$ parabolic Higgs bundles on $\BP^1$, such that the residue at each marked point lies in the closure of the minimal nilpotent orbit of $\mathfrak{sl}_r$, and whose underlying bundle is trivial. Under this identification, there is a natural Hitchin map
\begin{equation}
  \mathbf h: \CX^r_n(\alpha) \to B\nonumber
\end{equation}
where $B$ is an affine space of half the dimension of $\CX^r_n(\alpha)$. The components of $\mathbf h$ pairwise Poisson commute, and for ranks $2$ and $3$ the components are functionally independent.
\end{theorem}
%------------------------------------------------------------------------------

%------------------------------------------------------------------------------
\begin{remark}
There is an interesting distinction between the cases $r=2,\,3$,  and $r > 3$. In the $r=2$ case, the minimal nilpotent orbit is also the regular nilpotent orbit, and generic spectral curves are smooth. For $r=3$, the generic spectral curve is not smooth, but contains only transverse self-intersections for singularities. For $r > 3$, the singularities are worse and it is not clear \emph{a priori} whether the components of the Hitchin map are functionally independent. 
\end{remark}

%------------------------------------------------------------------------------

\noindent\textbf{Acknowledgements.} The problem of generalizing Konno's results was suggested to us by Tam\'as Hausel.  We would also like to thank Tom Baird, Philip Boalch, Peter Crooks, Andrew Dancer, Jacques Hurtubise, Lisa Jeffrey, and Alessia Mandini for useful discussions.  We express our gratitude to the organizers of the ``Workshop on Advances in Hyperk\"ahler and Holomorphic Symplectic Geometry'', held in March 2012 at the Banff International Research Station, and to the organizers of the ``Workshop on Moduli Spaces and their Invariants in Mathematical Physics'', held in June 2013 at the Centre de recherches math\'ematiques (CRM) in Montr\'eal, where some formative steps in this work were taken.

%=========================================================================================
% "Hyperpolygons and Hitchin Systems"
%
% Jonathan Fisher
% jonathan.fisher@uni-hamburg.de
%
% Steven Rayan
% stever@math.toronto.edu
% 
% TeX source for Section 2: "Construction"
%
%=========================================================================================

%=========================================================================================
\section{Polygons and Hyperpolygons}
%=========================================================================================

%=========================================================================================
\subsection{Star Quivers}
%=========================================================================================

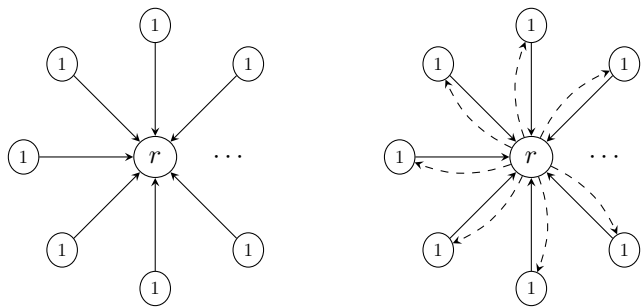
\begin{figure}[h!]
\begin{tikzpicture}
  \node[draw, ellipse] (Lsink) at (0,0) {$r$};
  \node[draw, ellipse] (Rsink) at (5,0) {$r$};
  % \dotnodebelow{sink}{c}{$2$};

  \foreach \a in {45, 90, ..., 359}
  {
    \node[draw, ellipse, scale=0.7] (Lv) at ($(Lsink)+(\a:1.75)$) {$1$};
    \path[-stealth] (Lv) edge (Lsink);

    \node[draw, ellipse, scale=0.7] (Rv) at ($(Rsink)+(\a:1.75)$) {$1$};
    \path[-stealth] (Rv) edge (Rsink);
    \path[-stealth, dashed] (Rsink) edge[bend left=20] (Rv);
  };

    \node (Ld) at ($(Lsink)+(0:1)$) {$\cdots$};
    \node (Rd) at ($(Rsink)+(0:1)$) {$\cdots$};

\end{tikzpicture}
\caption{Left: quiver whose moduli of representations is given by $\CP^r_n(\alpha)$. Right: the doubled quiver used to construct $\CX^r_n(\alpha)$.}
\label{fig-polygon-quiver}
\end{figure}

We take a moment now to formally introduce the main objects of this article: the moduli spaces of polygons and hyperpolygons in a complex projective space.  Begin by fixing a pair of integers, $r \geq 2$ and $n \geq r+1$.  \begin{definition}  A \emph{star-shaped quiver} (henceforth, a ``star quiver'') of \emph{rank} $r$ and \emph{twist} $n$ is a labelled finite directed graph $Q$ with $n+1$ nodes and $n$ arrows, such that\begin{itemize}\item exactly one node, the \emph{sink}, has label `$r$';\item each of the remaining nodes has label `$1$';\item from each node labelled `$1$', there departs exactly one arrow, whose tip is the sink,\end{itemize}
as depicted in Figure \ref{fig-polygon-quiver}.\end{definition}

The \emph{dimension vector} of $Q$ is the $(n+1)$-tuple of positive integers $\bd=(r,1,\dots,1)$, and this vector determines $Q$ uniquely as a star quiver. 
By $\Rep(Q)$, we mean the vector space of representations of $Q$ (respecting the labelling $\bd$), namely
\begin{equation}
  \Rep(Q) = \bigoplus_{i=1}^n \Hom(\BC, \BC^r) \iso \Hom(\BC^n, \BC^r) = \BC^{n\times r}.\nonumber
\end{equation}
We denote elements of $\Rep(Q)$ by $x$, and by the above isomorphism we identify $x$ with an $r \times n$ matrix, whose columns we denote by $x_1, \dots, x_n$.
There is a natural linear action of $U(r) \times (S^1)^n$ on $\Rep(Q)$; however the overall diagonal scalars act trivially. Quotienting by this subgroup, we obtain an effective action by the group
\begin{equation}
  G := P(U(r) \times (S^1)^n).\nonumber
\end{equation}
The Lie algebra $\fg$ of this group may be identified with $\mathfrak{su}_r \oplus \BR^n$, with center isomorphic to $\BR^n$. Hence a central element $\alpha \in \fgd$ may be identified with a vector in $\BR^n$ --- and we denote its components by $(\alpha_1, \dots, \alpha_n)$. We assume henceforth that the components of $\alpha$ are all positive. We call $\alpha$ the \emph{length vector}. 
If $S$ is a subset of $[n] = \set{1, \dots, n}$, put 
\begin{equation}
  \alpha_S:=\sum_{i\in S}{\alpha_i}.\nonumber
\end{equation} 

\begin{definition} We say that $\alpha \in \BR^n_{> 0}$ is \emph{generic} if, for all $0 \leq r' \leq r$ and $S \subset [n]$ such that $(r'-1)(\#S-r'-1) \geq 0$, we have
\begin{equation}
  r' \alpha_{[n]} - r \alpha_S \neq 0.\nonumber
\end{equation}

\end{definition}

%-----------------------------------------------------------------------------------------
\begin{example} In the case $r=2$ and assuming $\alpha_i > 0$ for all $i \in [n]$, the only non-vacuous conditions come from setting $r' = 1$, in which case we obtain the conditions $\alpha_{[n]} - 2\alpha_S \neq 0$, 
i.e. that every subset $S \subseteq [n]$ is either short ($\alpha_{[n]}-2\alpha_S > 0$) or long ($\alpha_{[n]}-2\alpha_S < 0$). This is exactly the genericity condition that appears for usual polygon spaces \cite{Klyachko, KonnoPolygon}.
\end{example}
%-----------------------------------------------------------------------------------------

The real and complex moment maps for the action of $G$ on $T^\ast \Rep(Q)$ are given by\footnote{Strictly speaking, these are certain scalar multiples of the real and complex moment maps, chosen to eliminate various factors of $-i$ and $\frac{1}{2}$ in the calculations of \S \ref{sec-cohomology}.}
\begin{align}
  \mur(x,y) &= ((xx^\ast - y^\ast y)_0, \normsq{x_1}-\normsq{y_1}, \dots, \normsq{x_n}-\normsq{y_n}) \nonumber\\
  \muc(x,y) &= ((xy)_0, y_1 x_1, \dots, y_n x_n)\nonumber
\end{align}
where $(\cdot)_0$ denotes the trace-free part of a matrix. The real moment map for the action of $G$ on $\Rep(Q)$ is given by the restriction of $\mur$ to $\Rep(Q)$.  We define the \emph{polygon space} $\CP_n^r(\alpha)$ to be the symplectic quotient
\begin{equation} \CP^r_n(\alpha) = \Rep(Q) \reda{\alpha} G. \nonumber\end{equation}
Similarly, we define the \emph{hyperpolygon space} $\CX_n^r(\alpha)$ to be the hyperk\"ahler quotient
\begin{equation} \CX^r_n(\alpha) = T^\ast \Rep(Q) \rreda{(\alpha,0)} G. \nonumber\end{equation}
By results of Nakajima \cite[Theorem 2.8, Corollary 4.2]{Nakajima94} we have the following.
%-----------------------------------------------------------------------------------------
\begin{theorem} 
If $\alpha$ is generic, then $G$ acts freely on $\mur^{-1}(\alpha) \cap \muc^{-1}(0)$, and $\CX_n^r(\alpha)$ is a smooth complete hyperk\"ahler manifold of complex dimension $2(r-1)(n-r-1)$. If non-empty, $\CP_n^r(\alpha)$ is a smooth compact K\"ahler manifold of complex dimension %$(r-1)(n-r-1)$
$\frac{1}{2} \dim \CX^r_n(\alpha)$. If $\CP^r_n(\alpha)$ is non-empty, there is a natural inclusion $T^\ast \CP_n^r(\alpha) \into \CX_n^r(\alpha)$ with dense image. Furthermore, the diffeomorphism type of $\CX^r_n(\alpha)$ is independent of $\alpha$.
\end{theorem}
%-----------------------------------------------------------------------------------------

\begin{remark} The cotangent bundle $T^\ast \Rep(Q)$ can be naturally identified with the space of representations of the \emph{doubled quiver}, i.e. the quiver with the same underlying set of vertices, but to which we add an arrow going in the opposite direction for every arrow of $Q$. This is depicted on the right side of Figure \ref{fig-polygon-quiver}, with the doubled arrows indicated by dashed lines.
\end{remark}
%-----------------------------------------------------------------------------------------

\subsection{Polygons}
In this section, we give two different interpretations of $\CP^r_n(\alpha)$ as a moduli spaces of polygons. First, observe that the real vector space $\mathfrak{su}_r$ is naturally Euclidean, with its Euclidean structure induced by the trace norm. Consider the map $\Rep(Q) \to \oplus_{i=1}^n\mathfrak{su}_r$ given by
\begin{equation}
  x_i \mapsto v_i := (x_i x_i^\ast)_0 = x_i x_i^\ast  - (\normsq{x_i}/r) \mathbf{1}_r\nonumber
\end{equation}
Imposing the real moment map equations $\normsq{x_i} = \alpha_i$, we find that
\begin{align}
  \normsq{v_i} &= \Tr(v_i v_i^\ast) = \Tr((x_i x_i^\ast  - (\normsq{x_i}/r) \mathbf{1}_r)(x_i x_i^\ast  - (\normsq{x_i}/r) \mathbf{1}_r)) \nonumber\\
 &= (1 - r^{-1}) \norm{x_i}^4 = (1 - r^{-1} ) \alpha_i^2,\nonumber
\end{align}
\noindent and thus the condition $\normsq{x_i}=\alpha_i$ is equivalent to $\normsq{v_i}=(1-r^{-1})\alpha_i^2$. Furthermore, the moment map equation $\sum_i (x_i x_i^\ast)_0$ gives $\sum_i v_i = 0$. Hence the collection $(v_1, \dots, v_n)$ consists of vectors of fixed lengths adding to zero, i.e. it is exactly the data of a closed polygon with fixed edge lengths\footnote{Note, however, that in general not every closed polygon with fixed edge lengths can be constructed in this way, since not every vector $v$ can be written as $(x x^\ast)_0$.}. See Figure \ref{fig-euclid-polygon}. Finally, since the action of $G$ on $\mathfrak{su}_r$ is the adjoint action, $G$ acts as a subgroup of the Euclidean group of $\mathfrak{su}_r$. Hence $\CP^r_n(\alpha)$ is a submanifold of the moduli space of polygons in $\mathfrak{su}_r$, with fixed edge lengths, considered up to the overall action of $SU(r)$.

\begin{figure}[h!]
\begin{tikzpicture}
  \node[fill=black, ellipse, scale=0.2] (1) at (0,0) {};
  \node[fill=black, ellipse, scale=0.2] (2) at (1,1) {};
  \node[fill=black, ellipse, scale=0.2] (3) at (2,3) {};
  \node[fill=black, ellipse, scale=0.2] (4) at (0,2.5) {};
  \node[fill=black, ellipse, scale=0.2] (5) at (-1,1) {};

    \path[-stealth] (1) edge node [below right] {$v_1$} (2);
    \path[-stealth] (2) edge node [right] {$v_2$} (3);
    \path[-stealth] (3) edge node [above left] {$v_3$} (4);
    \path[-stealth] (4) edge node [above left] {$v_4$} (5);
    \path[-stealth] (5) edge node [below left] {$v_5$} (1);

\end{tikzpicture}
\caption{Euclidean polygon determined by a point in a polygon space.}
\label{fig-euclid-polygon}
\end{figure}
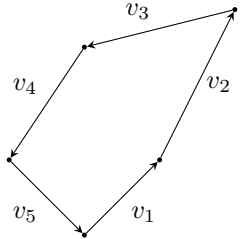

\begin{example} Consider the special case $r=2$. In this case, $\mathfrak{su}_2 \iso \mathfrak{so}_3 \iso \BR^3$, and furthermore, the map $x_i \to (x_i x_i^\ast)_0$ surjects onto the 2-sphere. This realizes $\CP^2_n(\alpha)$ as the moduli space of $n$-gons in $\BR^3$, with fixed edge lengths $(\alpha_1/2, \dots, \alpha_n / 2)$, considered up to the overal rotational action of $PSU(2) \iso SO(3)$. This is exactly the special case considered by Klyachko \cite{Klyachko}.

\end{example}

%-----------------------------------------------------------------------------------------

\begin{remark}
This construction can be generalized in the following fashion. Let $\CO_1, \dots, \CO_n$ be coadjoint orbits in the dual of some Lie algebra $\fg$. The moment map for the coadjoint action of $G$ on the product $\CO_1 \times \dots \times \CO_n$ is given by $\mu(x_1, \dots, x_n) = x_1 + \dots + x_n$.  Consequently, the symplectic quotient $(\CO_1 \times \dots \times \CO_n) \red G$ may be regarded as the moduli space of closed $n$-gons in $\fgd$, considered up to rotations by $G$, whose edges lie in a set of fixed coadjoint orbits.
\end{remark}

%-----------------------------------------------------------------------------------------

Finally, we give a different interpretation of the space $\CP^r_n(\alpha)$. Note that
\begin{equation}
  \Rep(Q) \reda{\alpha} (S^1)^n \iso \BP^{r-1} \times \dots \times \BP^{r-1},\nonumber
\end{equation}
where the $i$th factor of $\BP^{r-1}$ is equipped with $\alpha_i$ times the Fubini-Study form. Any point in $\prod_{i=1}^n \BP^{r-1}$ can be thought of as giving the vertices of a polygon in $\BP^{r-1}$, and the diagonal action of $PGL(r)$ acts by projective transformations on this polygon.  From King \cite[\S 6]{King}, we have:

%-----------------------------------------------------------------------------------------
\begin{theorem} Suppose that $\alpha$ is integral, i.e. $\alpha = d\chi$ for some character $\chi: (S^1)^n \to S^1$. Then as a complex manifold, $\CP^r_n(\alpha)$ may be naturally identified with the moduli space of $\alpha$-semistable $n$-gons in $\BP^{r-1}$, considered up to overall projective equivalence.
\end{theorem}
%-----------------------------------------------------------------------------------------

%-----------------------------------------------------------------------------------------
\begin{remark} A related moduli space of \emph{twisted $n$-gons} was constructed by Khesin and Soloviev \cite{KhSo2013}, as the natural configuration space for a generalized pentagram map.
\end{remark}
%-----------------------------------------------------------------------------------------

%=========================================================================================
\subsection{Nilpotent Orbits}
%=========================================================================================
To help motivate the connection with Hitchin systems, we recall some basic facts about nilpotent orbits in $\mathfrak{sl}_r$. Identifying $\mathfrak{sl}_r$ with the space of $r \times r$ traceless complex matrices, the Lie-Poisson structure is given explicitly by the relation
\begin{equation}
  \{x_{ij}, x_{kl}\} = \delta_{jk} x_{il} - \delta_{il} x_{kj},\nonumber
\end{equation}
where $x_{ij}$ denotes the corresponding component of an $r \times r$ matrix. The symplectic leaves of this Poisson structure are exactly the adjoint orbits. Of particular interest are the nilpotent orbits, i.e. the adjoint orbits of a nilpotent $r \times r$ matrix. The Zariski closures of nilpotent orbits are singular affine Poisson varieties, which admit resolutions by the cotangent bundles of (partial) flag varieties\footnote{Outside of type $A$, this is not always the case.}.

There is a natural partial order on nilpotent orbits, defined by
\begin{equation}
  \CO_1 \leq \CO_2 \ \textrm{if and only if}\ \CO_1 \subseteq \overline{\CO}_2.\nonumber
\end{equation}
With respect to this partial order, there are two distinguished orbits: the maximal, regular orbit $\Oreg$, and the minimal orbit $\Omin$. These orbits are completely characterized by their Jordan form representatives, which include
\begin{equation}
   \begin{bmatrix} 0 & 1 & & &  \\   & 0 & 1 & & \\ & & \ddots & \ddots & & \\ & & & 0 & 1 \\  &  &   &  & 0 \end{bmatrix}\ \ \ \mbox{and} \ \ \
   \begin{bmatrix} 0 & 1 & & \\ 0 & 0 & & \\  & & \ddots& \\ & & & 0\end{bmatrix}\nonumber
\end{equation}
\noindent in $\Oreg$ and $\Omin$, respectively.

%-----------------------------------------------------------------------------------------
\begin{proposition} The closure $\Oregcl$ of the regular nilpotent orbit consists of the nilpotent cone
\begin{equation}
  \Oregcl = \{ x \in \mathfrak{sl}_r \suchthat x^r = 0 \},\nonumber
\end{equation}
while the closure $\Omincl$ of the minimal nilpotent orbit is the variety
\begin{equation}
  \Omincl = \{ x \in \mathfrak{sl}_r \suchthat x^2 = 0, \ \rk(x) \leq 1 \}.\nonumber
\end{equation}
\end{proposition}
%-----------------------------------------------------------------------------------------

Now, consider $T^\ast \BP^{r-1}$ as the hyperkähler quotient of $T^\ast \BC^r$, with coordinates $x \in \BC^r$ and $y \in \BC^r$, with $x$ a column vector and $y$ a row vector. The left action of $GL(r)$ on $(x,y)$ given by $g \cdot (x,y) = (gx, yg^{-1})$ descends to an action on $T^\ast \BP^{r-1}$, with complex moment map given by
\begin{equation}
  \muc(x,y) = xy \in \mathfrak{gl}_r.\nonumber
\end{equation}
Clearly, $\rk(\muc(x,y)) \leq 1$, and we see from the hyperkähler quotient construction that $\muc(x,y)^2 = (xy)(xy) = x(yx)y = 0$. Hence $\muc$ takes values in $\Omincl$.
%-----------------------------------------------------------------------------------------
\begin{proposition} \label{prop-nilpotent-resolution}
The complex moment map $\muc: T^\ast \BP^{r-1} \to \Omincl$ is Poisson, surjective, and generically one-to-one.
\end{proposition}
%-----------------------------------------------------------------------------------------
\begin{proof} That $\muc$ is Poisson follows, essentially, by definition of the moment map. To see that it is surjective and generically one-to-one, note that $\Omincl = \Omin \cup \{0\}$. It is easy to check that $\muc^{-1}(0) = \BP^{r-1} \subset T^\ast \BP^{r-1}$. Now suppose $\phi \in \Omin$. Since $\phi$ is rank 1, it can be factored as $\phi = xy$ for some row vector $x$ and column vector $y$. These vectors are unique up to the equivalence $(x,y) \sim (\lambda x, \lambda^{-1} y)$. Moreover, the nilpotency condition $\phi^2=0$ implies that $0 = (xy)^2 = (xy)(xy) = (xy) \phi$, and since $\phi$ is not identically zero, we find that $xy=0$. Hence, $[x,y]$ determines a unique point in $T^\ast \BP^{r-1}$.
\end{proof}
%-----------------------------------------------------------------------------------------

%=========================================================================================
% "Hyperpolygons and Hitchin Systems"
%
% Jonathan Fisher
% jonathan.fisher@uni-hamburg.de
%
% Steven Rayan
% stever@math.toronto.edu
% 
% TeX source for Section 1: "Cohomology"
%
%=========================================================================================

%==============================================================================
\section{Cohomology} \label{sec-cohomology}
%==============================================================================

%=============================================================================
\subsection{Kirwan Surjectivity} \label{sec-surjectivity}
%=============================================================================
We begin by recalling that a Kähler manifold $M$ is \emph{circle compact}\footnote{In the algebro-geometric context, the analogous condition is sometimes called \emph{semiprojectivity}, see e.g. \cite{HRV2013}} if it admits a Hamiltonian $S^1$-action, such that the following conditions are satisfied:
\begin{enumerate}
  \item $M^{S^1}$, the fixed point set, is compact;
  \item the moment map is proper and bounded below.
\end{enumerate}
Given a circle compact manifold $M$, we may construct a compactification $\overline{M}$ using the symplectic cut construction \cite{Lerman} as follows. Let $\Lambda \gg 0$ and define $\overline{M}$ to be
\begin{equation}
  \overline{M} = (M \times \BC) \reda{\Lambda} S^1,\nonumber
\end{equation}
where the $S^1$ acts diagonally on $M \times \BC$. At worst, $\overline{M}$ has orbifold singularities, and by elementary Morse theory with the moment map, its orbifold diffeomorphism type is independent of $\Lambda$ for $\Lambda$ sufficiently large. We then have a natural inclusion $M \into \overline{M}$. We denote by $\partial M$ the boundary divisor
\begin{equation}
  \partial M = \overline{M} \setminus M \iso M \reda{\Lambda} S^1,\nonumber
\end{equation}
which itself has at worst orbifold singularities.

%-----------------------------------------------------------------------------------------
\begin{proposition} \label{prop-a} Suppose that $M$ has a Hamiltonian $S^1 \times G$-action such that the $S^1$-action makes $M$ circle compact. Then the restriction $H^\ast_G(\overline{M}) \to H^\ast_G(M)$ is surjective. Furthermore, we have the relation
\begin{equation}
  P_t^G(M) = P_t^G(\overline{M}) - t^2 P_t^G(\partial M).\nonumber
\end{equation}
\end{proposition}
%-----------------------------------------------------------------------------------------
\begin{proof} The cut construction endows $\overline{M}$ with a residual Hamiltonian $S^1$-action. Moreover, the fixed-points of this action are given by
\begin{equation}
  \overline{M}^{S^1} = M^{S^1} \sqcup \partial M.\nonumber
\end{equation}
The result then follows by ordinary Morse theory with the $S^1$-moment map.

% Stever: I added the qualifier "ordinary" here.

\end{proof}
%-----------------------------------------------------------------------------------------

%-----------------------------------------------------------------------------------------
\begin{proposition} \label{prop-b}
Suppose that $M$ has a Hamiltonian $(S^1 \times G \times K)$-action such that the $S^1$-action makes $M$ circle compact. Then the $K$-equivariant Kirwan map $H_{G \times K}^\ast(M) \to H^\ast_K(M \red G)$ is surjective.
\end{proposition}
%-----------------------------------------------------------------------------------------
\begin{proof} Since the $S^1 \times G$ moment map is proper, it follows from equivariant Morse theory with its norm-square \cite{Kirwan} that the map 
\begin{equation}
  H_{S^1 \times G \times K}^\ast(M) \to H^\ast_K(\overline{M \red G})\nonumber
\end{equation}
\noindent is surjective. By Proposition \ref{prop-a}, the restriction $H^\ast_K(\overline{M \red G}) \to H^\ast_K(M \red G)$ is surjective. Hence we have a surjection $H_{S^1 \times G \times K}^\ast(M) \to H^\ast_K(M \red G)$. However, it is easy to check (using the cut construction) that this map is essentially the Kirwan map, with the $S^1$ factor acting trivially. Hence we deduce that $H_G^\ast(M) \to H^\ast(M \red G)$ is surjective.
\end{proof}
%-----------------------------------------------------------------------------------------

We use this construction to study the cohomology of hyperpolygon spaces.  To begin, let $T$ be the quotient of $(S^1)^n$ by the diagonal subgroup.  Reduction in stages gives
\begin{align}
  \CP^r_n(\alpha) &\iso (T^\ast \Rep(Q) \rred U(r)) \rreda{(\alpha,0)} T\nonumber \\ &\iso T^\ast Gr(r, n) \rreda{(\alpha,0)} T.\nonumber
\end{align}
Denote by $\CG^r_n(\alpha)$ the non-compact symplectic quotient
\begin{equation}
  \CG^r_n(\alpha) = T^\ast Gr(r,n) \reda{\alpha} T,\nonumber
\end{equation}
and note that there is a natural inclusion of $\CX^r_n(\alpha)$ into $\CG^r_n(\alpha)$ as the zero set of the complex $T$-moment map. The hyperk\"ahler Kirwan map $H^\ast(BG) \to H^\ast(\CX^r_n(\alpha))$ factors as
\begin{equation}
  H^\ast(BG) \to H^\ast_T(\CG^r_n) \to H^\ast(\CX^r_n(\alpha)),\nonumber
\end{equation}
and the first map is surjective by Proposition \ref{prop-b} above. 

%-----------------------------------------------------------------------------------------
\begin{theorem} \label{thm-homotopy} 
The natural inclusion $\CX^r_n(\alpha) \into \CG^r_n(\alpha)$ induces an equivalence in both ordinary and $S^1$-equivariant cohomology.
\end{theorem}
%-----------------------------------------------------------------------------------------
\begin{proof} First we show that that $\CX^r_n(\alpha)^{S^1} = \CG^r_n(\alpha)^{S^1}$. Suppose $[x,y] \in \CG^r_n(\alpha)^{S^1}$. Then there exists a homomorphism $\phi: S^1 \to T$ such that for all $s \in S^1$, $(x, sy) = (\phi(s) x, y \phi^{-1}(s))$. Since $T$ is abelian, we find that
\begin{equation}
  s \mu_c(x,y) = \mu_c(x, sy) = \mu_c(\phi(s) x, y \phi^{-1}(s)) = \mu_c(x,y),\nonumber
\end{equation}
from which it follows that $\mu_c(x,y) = 0$ and $[x,y] \in \CX^r_n(\alpha)$. Now consider the isotropy representation of $S^1$ on $T_x \CG^r_n(\alpha)$. It fits into the short exact sequence

\begin{equation}
 0 \longrightarrow T_x \CX^r_n(\alpha) \longrightarrow T_x \CG^r_n(\alpha) \longrightarrow N_x \CX^r_n(\alpha) \longrightarrow 0.\nonumber
\end{equation}
Since the complex moment map is of weight $1$ with respect to the $S^1$-action, we deduce that $N_x \CX^r_n(\alpha)$ is a representation consisting of strictly positive weights. Hence, we find that the $S^1$-moment map, when restricted to $\CX^r_n(\alpha)$, has the same critical sets with the same Morse indices as it does on $\CG^r_n(\alpha)$.

Consider the Morse stratifications of $\CG^r_n(\alpha)$ and $\CX^r_n(\alpha)$. We denote by $\CG_C^\pm$ and $\CX_C^\pm$ the super- and sub-level sets, respectively. By the Atiyah-Bott lemma \cite[13.4]{YangMillsRiemannSurface} the associated Thom-Gysin sequence splits into short exact sequences, and so for each critical set $C$ we obtain the commutative diagram below, where the downward arrows are induced by the natural inclusions.
\begin{equation} \label{diagram-thom-restriction}
  \begin{tikzpicture}[baseline=(current bounding box.center)]
    \matrix (m) [matrix of math nodes,row sep=3em,column sep=3em,minimum width=2em]
    {
      0 & H_{S^1}^{k-\lambda_C}(C) & H_{S^1}^k(\CG_C^+) & H_{S^1}^k(\CG_C^-) & 0\\
      0 & H_{S^1}^{k-\lambda_C}(C) & H_{S^1}^k(\CX_C^+) & H_{S^1}^k(\CX_C^-) & 0\\
    };
    
    \path[-stealth] (m-1-2) edge node [left] {$\iso$} (m-2-2);
    \path[-stealth] (m-1-3) edge (m-2-3);
    \path[-stealth] (m-1-4) edge (m-2-4);
    
    \foreach \n [count=\np from 2] in {1,2,3,4}
    {
      \path[-stealth] (m-1-\n) edge (m-1-\np);
      \path[-stealth] (m-2-\n) edge (m-2-\np);
    };
  \end{tikzpicture}
\end{equation}
Let $C_0$ be the critical set corresponding to the absolute minimum of the moment map. In this case we have $\CG_{C_0}^- = \CX_{C_0}^- = \emptyset$, and so by the above diagram we have $H_{S^1}^k(\CG_{C_0}^+) \iso H_{S^1}^k(\CX_{C_0}^+)$. Now assume by induction that the rightmost verical arrow of diagram \eqref{diagram-thom-restriction} is an isomorphism. By a standard diagram chase, we find that $H_{S^1}^k(\CG_C^+) \iso H_{S^1}^k(\CX_C^+)$. Continuing this way by induction, we find that $H_{S^1}^\ast(\CG^r_n(\alpha)) \iso H_{S^1}^\ast(\CX^r_n(\alpha))$. By equivariant formality, this implies that we also have $H^\ast(\CG^r_n(\alpha)) \iso H^\ast(\CX^r_n(\alpha))$.
\end{proof}
%----------------------------------------------------------------------------------------

Combining Theorem \ref{thm-homotopy} with the discussion preceding it, we thereby have a complete proof of

%-----------------------------------------------------------------------------------------

\begin{theorem}\label{main-thm-kirwan}  The natural hyperk\"ahler Kirwan map
\begin{equation}
  \kappa: H^\ast(B(S(U(r) \times (S^1)^n))) \to H^\ast(\CX^r_n(\alpha))\nonumber
\end{equation} 

\noindent{is surjective.}

\end{theorem}

%-----------------------------------------------------------------------------------------

\begin{remark}
The kernel of the $S^1$-equivariant Kirwan map may be computed by the abelianization technique of Hausel-Proudfoot \cite{HauselProudfoot}, which reduces this problem to an explicit calculation in an associated hypertoric variety. The cohomology rings $H^\ast(\CX^2_n(\alpha))$ are well known; we plan to investigate $H^\ast(\CX^r_n(\alpha))$ in future work. It is clear from preliminary calculations that the kernel of the Kirwan map is significantly more complicated for $r \geq 3$ than it is for $r=2$.
\end{remark}

%=========================================================================================
\subsection{Compactified Morse Theory} \label{sec-morse-theory} 
%=========================================================================================

An immediate consequence of Theorem \ref{thm-homotopy} is that $P_t(\CX^r_n(\alpha)) = P_t(\CG^r_n(\alpha))$. We would like to compute the latter side of this identity by using equivariant Morse theory with the norm-square of the $T$-moment map. Unfortunately, since $T^\ast Gr(r,n)$ is non-compact we have to take some care to justify the Morse theory. By the Proposition \ref{prop-a}, we have
\begin{equation} \label{eqn-poincare-difference}
  P_t(\CG^r_n(\alpha)) =   P_t(\overline{\CG^r_n(\alpha)}) - t^2 P_t(\partial \CG^r_n(\alpha)).
\end{equation}
Note that both $\overline{\CG^r_n(\alpha)}$ and $\partial \CG^r_n(\alpha)$ may be constructed as symplectic quotients by $T$ of the compact manifolds $\overline{M}$ and $\partial M$, where $M = T^\ast Gr(r,n)$. Hence, we may compute their Poincar\'e polynomials by equivariant Morse theory with the norm-squares of their respective moment maps. Both terms in the right hand side of the above equation contain contributions from critical sets at infinity, which have a complicated dependence on $\alpha$. The goal of this section is to show that these terms exactly cancel.

Since the critical set for the norm-square of the moment map on $\partial M$ is contained in the critical set for the norm-square of the moment map on $\overline{M}$, the connected components $C$ of the critical set on $\partial M$ fall into three possible types:
\begin{itemize}
  \item those of the form $C = C' \cap \partial M$, where $C'$ is a component of the critical set on $\overline{M}$ not contained in $\partial M$;
  \item those for which the Morse index on $\overline{M}$ is strictly larger than the Morse index on $\partial M$; and
  \item those for which the Morse index of $C$ on $\partial{M}$ is the same as its index on $\overline{M}$.
\end{itemize}
Respectively, we call these types \emph{ordinary}, \emph{boundary}, and \emph{anomalous}. The dichotomy between the boundary and anomalous critical sets is depicted schematically in Figure \ref{fig-anomalous-distinction}. Intuitively, the anomalous critical sets are the possible limit points of gradient trajectories on $M$ that escape to infinity.
%-----------------------------------------------------------------------------------------
\begin{figure}[h]
\centering
\begin{tikzpicture}[auto]

\tikzset{->-/.style={decoration={
  markings,
  mark=at position #1 with {\arrow{>}}},postaction={decorate}}}

  \coordinate (x) at ( 4, 0);
  
  \coordinate (y) at ( 0, 2);
  \coordinate (ry) at (2, 0);

  \foreach \n in {0, 1}
  {
    \coordinate (c1-\n) at ( $\n*(x)+\n*0.4*(x)$ );
    \coordinate (c2-\n) at ( $(c1-\n)+(x)$ );
    \coordinate (c3-\n) at ( $(c1-\n)+(x)+(y)$ );
    \coordinate (c4-\n) at ( $(c1-\n)+(y)$ );
    \coordinate (c5-\n) at ( $0.5*(c2-\n)+0.5*(c3-\n)$ ) {};

    \coordinate (c6-\n) at ( $(c5-\n) - 0.5*(ry)$ );

    \coordinate (c7-\n) at ( $0.7*(c1-\n)+0.3*(c2-\n)+0.5*(y)$ );
    \coordinate (c8-\n) at ( $(c2-\n)-0.2*(y)$ );

    \coordinate (c9-\n) at ( $(c2-\n)+0.5*(y)+0.1*(x)$ );

    \coordinate (c10-\n) at ( $(c5-\n) + (135:1)$ );
    \coordinate (c11-\n) at ( $(c5-\n) + (225:1)$ );

    \draw[fill=black!20!white,black!20!white] (c1-\n) rectangle (c3-\n);

    \node[fill=black, ellipse, scale=0.4] (n) at (c5-\n) {};

   \draw[->-=0.5,thick] (c2-\n) to (c5-\n);
    \draw[->-=0.5,thick] (c3-\n) to (c5-\n);

    \node (n1) at (c7-\n) {$M$};
    \node (n2) at (c8-\n) {$\partial M$};
    \node (n3) at (c9-\n) {$C$};
  }

   \draw[->-=0.5,thick] (c2-0) to [bend right=45] (c6-0);
  \draw[->-=0.5,thick] (c3-0) to [bend left=45] (c6-0);
  \draw[->-=0.5,thick] (c5-0) to [bend left=0] (c6-0);

  \draw[->-=0.5,thick] (c10-1) to (c5-1);
  \draw[->-=0.5,thick] (c11-1) to (c5-1);
  \draw[->-=0.5,thick] (c6-1) to (c5-1);

\end{tikzpicture} 
\caption{Left: The normal direction lies in the negative normal bundle to the critical set, causing an increase in the Morse index. Right: The normal direction is in the positive normal bundle of the critical set.}
\label{fig-anomalous-distinction}
\end{figure}
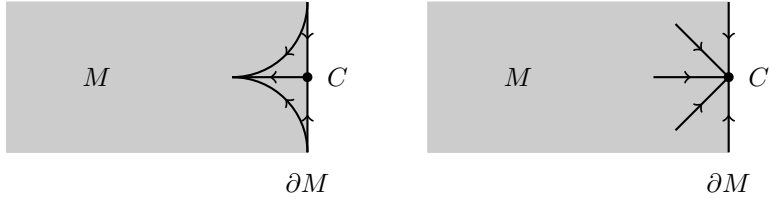
%-----------------------------------------------------------------------------------------

Since $\partial M$ has real codimension 2 in $\overline{M}$ (as follows from the cut construction), the difference in the Morse indices of boundary critical sets is precisely 2. Hence, the contributions from boundary terms in the right hand side of equation \eqref{eqn-poincare-difference} automatically cancel. On the other hand, for every ordinary critical set $C$ the right hand side of equation \eqref{eqn-poincare-difference} will contain a term of the form
\begin{equation}
  P^T_t(C) - t^2 P^T_t(C \cap (\partial M) ),\nonumber
\end{equation}
which by the same argument as in the proof of Proposition \ref{prop-a}, is equal to 
\begin{equation}
  P^T_t(C \setminus (C \cap \partial M)) = P^T_t(C \cap M),\nonumber
\end{equation} that is, equal to the equivariant Poincar\'e series of the uncompactified critical set. Consequently, the only possible contribution from critical sets at infinity come from the anomalous types. The goal of the remainder of this section is to prove the following.

%-----------------------------------------------------------------------------------------
\begin{lemma} \label{main-lemma}
For a sufficiently large cut parameter $\Lambda$, there are no anomalous critical sets.
\end{lemma}
%-----------------------------------------------------------------------------------------

Before we continue, let us give an explicit description of the cut compatficitation $\overline{M}$. It consists of all equivalence classes $[x, y, z]$, with $x$ an $r \times n$ matrix, $y$ an $n \times r$ matrix, and $z$ a complex number, $[x, y, z] \sim [g x, syg^{-1}, sz]$ for $(g,s) \in U(r) \times S^1$, subject to the moment map equations
\begin{align}
  x x^* - y^* y &= (\alpha_{[n]} / r) \mathbf{1}_r \nonumber\\
  xy &= 0 \nonumber\\
  \normsq{y} + \normsq{z} &= \Lambda,\nonumber
\end{align}
where $\Lambda > 0$ is the cut parameter.

Now recall from \cite{Kirwan} that the components of the critical set of $\norm{\mu}^2$ for a moment map $\mu$ on a symplectic manifold $M$ associated to the action of a torus $T$ are all of the form
\begin{equation}
  C_\beta = Z_\beta \cap M^{T_\beta},\nonumber
\end{equation}
where $Z_\beta := \mu^{-1}(\beta)$ and $T_\beta$ is the closure in $T$ of $\{\exp(t\beta) \ | \ t \in \BR\}$. Let us return to the specific case of the $T$-action on $\overline{M}$. Suppose that $K \subset (S^1)^n$ is a sub-torus of the standard maximal torus $(S^1)^n \subset U(n)$. If we have $[x,y,z] \in \overline{M}^K$, then there exists a homomorphism $(g,h):K \to U(r) \times S^1$ such that
\begin{align}
  xk &= g(k) x \nonumber\\
  k^{-1}y &= h(k) y g(k)^{-1} \nonumber\\
  z &= h(k) z \nonumber
\end{align}
Working up to conjugation in $U(r)$, we can assume $g(k)$ takes the normal form
\begin{equation}
  \begin{bmatrix}
    k^{\sigma_1} \mathbf{1}_{i_1} &  & 0\nonumber \\
     & \ddots & \nonumber \\
    0 &  & k^{\sigma_d} \mathbf{1}_{i_d}\nonumber
  \end{bmatrix}
\end{equation}
where $\sigma_1, \dots, \sigma_d$ are distinct weights of $K$ and $\mathbf{1}_i$ denotes the $i \times i$ identity matrix.  Without loss of generality, assume that $i_1 \leq \dots \leq i_d$. This picks out a partition $\lambda$ of $r$, where $r = i_1 + \dots + i_d$. Let $\rho_1, \dots, \rho_n$ denote the weights of $K \into (S^1)^n$. Under the action
\begin{equation}
  k \cdot (x, y, z) = ( g(k)^{-1} x k, h(k) k^{-1} y g(k), h(k) z)\nonumber
\end{equation}
we find that
\begin{align}
  x_{ij} &\mapsto k^{-\sigma_i+\rho_j} x_{ij}\nonumber \\
  y_{ji} &\mapsto k^{\nu+\sigma_i-\rho_j} y_{ji} \nonumber\\
  z &\mapsto k^\nu z\nonumber
\end{align}
where $h(k) = k^\nu$, and $x_{ij}$ and $y_{ji}$ denote components of the matrices $x$ and $y$.

%-----------------------------------------------------------------------------------------
\begin{proof}[Proof of Lemma \ref{main-lemma}] Consider a subtorus $K = T_\beta$ given by the weights $\beta=(\sigma, \rho, \nu)$ as above. The corresponding critical set is anomalous precisely when $\nu > 0$, so we must show that $\nu\leq0$. We can assume that $\nu \neq 0$, since $\nu=0$ corresponds to an ordinary critical set. In this case, we find immediately that $z=0$. The weights $(\sigma, \rho)$ decompose the quiver pictured in Figure \ref{fig-polygon-quiver} as follows. The central vertex (representing $\BC^r$) splits up into $\ell$ vertices, each labelled by the weight $\sigma_i$ and of dimension $d_i$, and the valent vertices are labelled by the weights $\rho_j$. We draw an arrow from $\rho_j$ to $\sigma_i$, corresponding to a non-zero block of $x$, if and only if $\sigma_i=\rho_j$. Similarly, we draw an arrow from $\sigma_i$ to $\rho_j$, representing a non-zero block of $y$, if and only if $\rho_j=\sigma_i+\nu$.
A typical quiver representing a critical set is pictured in Figure \ref{fig-anomalous}.

\begin{figure}[h]
\centering
\begin{tikzpicture}[auto]

  \coordinate (dx) at (0.5, 0);
  \coordinate (dy) at (0, 0.7);

  \node[fill=black, ellipse, scale=0.2] (1) at (0, 1) {};
  \node[fill=black, ellipse, scale=0.2] (2) at (2, 1) {};
  \node[fill=black, ellipse, scale=0.2] (3) at (4, 1) {};

  \node[fill=black, ellipse, scale=0.2] (4) at (0, -1) {};
  \node[fill=black, ellipse, scale=0.2] (5) at (2, -1) {};
  \node[fill=black, ellipse, scale=0.2] (6) at (4, -1) {};

  \foreach \n in {1,2,3, 4, 5, 6}
  {
    \node[fill=black, ellipse, scale=0.2] (n1) at ( $(\n)+(dx)+(dy)$ ) {};    
    \node[fill=black, ellipse, scale=0.2] (n2) at ( $(\n)-(dx)+(dy)$ ) {};    
    \node[fill=black, ellipse, scale=0.2] (n3) at ( $(\n)-(dx)-(dy)$ ) {};    
    \node[fill=black, ellipse, scale=0.2] (n4) at ( $(\n)+(dx)-(dy)$ ) {};    
    \path[-stealth] (n1) edge (\n);
    \path[-stealth] (n2) edge (\n);
    \path[-stealth, dashed] (\n) edge (n3);
    \path[-stealth, dashed] (\n) edge (n4);
  };

  \foreach \n in {1,2, 5}
  {
    \node[fill=black, ellipse, scale=0.2] (u) at ( $(\n)+(1,0)$ ) {};    
    \node[fill=black, ellipse, scale=0.2] (u2) at ( $(\n)+(2,0)$ ) {};    

    \path[-stealth, dashed] (\n) edge (u);
    \path[-stealth] (u) edge (u2);
  };

\end{tikzpicture}
\caption{A typical decomposition of a star quiver.}
\label{fig-anomalous}
\end{figure}
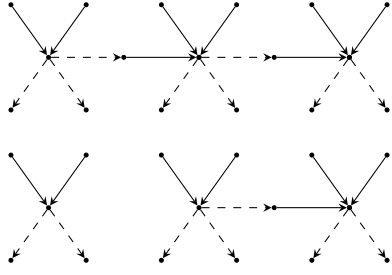

Let us consider a single connected component of the decomposed quiver. The horizontal arrows correspond to maps $\sigma_i \to \rho_j \to \sigma_k$, in which the first arrow corresponds to a block of $y$ and the second to a block of $x$. The condition on the blocks of $y$ implies that $\rho_j=\sigma_i+\nu$, and the condition on the blocks of $x$ implies that $\sigma_k=\rho_j$. Hence we find that $\sigma_k=\sigma_i+\nu$. Therefore a connected component of the decomposed quiver can be labelled by a maximal subset of $\set{\sigma_1, \dots, \sigma_d}$ of the form $\set{\sigma, \sigma+\nu,\dots,\sigma+m\nu}$ ($m$ is allowed to be $0$). At each node of the chain, label the arrows as in Figure \ref{fig-quiver-induction}. The real moment map at each node is
\begin{equation}
  x_i x_i^\ast + u_i u_i^\ast - y_i^\ast y_i - v_i^\ast v_i = (\alpha / r) \mathbf{1}_{d_i},\nonumber
\end{equation}
while the conditions that $\mu_T=\beta$ lie in the stabilizer of $T_\beta$ imply
\begin{align}
  \normsq{u_i} - \normsq{v_{i-1}} &= \alpha_i + (n_i^u)(\sigma + i \nu) \nonumber\\
  \normsq{x_i}  &= \alpha_i+ (n_i^x)(\sigma + i \nu) \nonumber\\
  -\normsq{y_i} &= \alpha_i + (n_i^y)(\sigma + (i+1) \nu),\nonumber
\end{align}
where $n_i^x, n_i^u$, and $n_i^y$ denote the sizes of the blocks (i.e. $x_i$ is $d_i \times n_i^x$, $y_i$ is $n_i^y \times d_i$, and so on).
Using the first equation, we have
\begin{align}
  (d_i/r)\alpha &= \sum_i (\normsq{x_i}+\normsq{u_i}-\normsq{y_i}-\normsq{v_i})\nonumber \\
    &= \sum_i (\normsq{x_i}-\normsq{y_i}+\normsq{u_i}-\normsq{v_{i-1}})\nonumber \\
    &\sim \left(\sum_i (n_i^x+n_i^u+n_i^y)\right) \sigma + \left(\sum_i (i n_i^x+(i+1)n_i^y+in_i^u)\right)\nonumber, \nu
\end{align}
where $\sim$ denotes equality up to terms linear in $\alpha$.  From this, we find $A \sigma + B \nu \sim 0$, where
\begin{align}
  A &= \sum_i (n_i^x+n_i^y+n_i^y),\nonumber \\
  B &= \sum_i (in_i^x+(i+1) n_i^y + i n_i^u).\nonumber
\end{align}

\begin{figure}[h]
\centering
\begin{tikzpicture}[auto]

  \node[fill=black, ellipse, scale=0.3] (s) at (0,0) {};
  \node[fill=black, ellipse, scale=0.2] (x) at (0,1.5) {};
  \node[fill=black, ellipse, scale=0.2] (y) at (0,-1.5) {};
  \node[fill=black, ellipse, scale=0.2] (u) at (-1.5,0) {};
  \node[fill=black, ellipse, scale=0.2] (v) at (1.5,0) {};
  \node[fill=black, ellipse, scale=0.2] (u1) at (3,0) {};
  \node[fill=black, ellipse, scale=0.2] (v1) at (-3,0) {};

  \path[-stealth] (x) edge node [right] {$x_i$} (s);
  \path[-stealth] (u) edge node [above] {$u_i$} (s);
  \path[-stealth,dashed] (s) edge node [left] {$y_i$} (y);
  \path[-stealth,dashed] (s) edge node [below] {$v_i$} (v);

  \path[-stealth] (v) edge node [below] {$u_{i+1}$} (u1);
  \path[-stealth,dashed] (v1) edge node [above] {$v_{i-1}$} (u);

\end{tikzpicture}
\caption{A single node in a chain of a decomposed quiver.}
\label{fig-quiver-induction}
\end{figure}
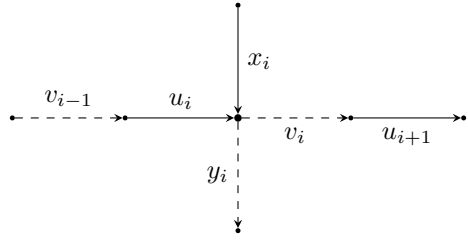

Furthermore, we can use the equations above to solve for $\normsq{v_i}$ in terms of $\normsq{v_{i-1}}$, which allows us to compute the $\normsq{v_i}$ inductively. We find
\begin{equation}
  \normsq{v_i} \sim \sum_{j=0}^i (A_j \sigma + B_j \nu),\nonumber
\end{equation}
where $A_i = n_i^x+n_i^y+n_i^u$ and $B_i=i n_i^x+(i+1)n_i^y+in_i^u$. Hence
\begin{align}
  \sum_i(\normsq{y_i} + \normsq{v_i})
  &\sim \sum_{i=0}^N \left(-n_i^y(\sigma + (i+1) \nu) + \sum_{j=0}^i (A_i \sigma + B_j \nu)\right) \nonumber\\
  &\sim -\sum_{i=0}^N n_i^y(\sigma + (i+1) \nu) + \sum_{i=0}^N (N+1-i)(A_i \sigma + B_i \nu), \nonumber\\
  &\sim C \sigma + D \nu,\nonumber
\end{align}
where
\begin{align}
  C &\sim \sum_{i=0}^N \left( (N+1-i)(n_i^x+n_i^u) + (N-i) n_i^y \right) \nonumber\\
  D &\sim \sum_{i=0}^N \left( (N+1-i)(in_i^x+in_i^u) + (N-i)(i+1) n_i^y \right). \nonumber
\end{align}
Now, using $A \sigma + B \nu \sim 0$, we find that
\begin{equation}
  \normsq{y} + \normsq{v} - \nu \sim C \sigma + (D-1) \nu \sim \left(\frac{AD-BC-A}{A} \right) \nu.\label{eqn-normsquares}
\end{equation}
It is an easy exercise, which we leave to the reader, to verify that $AD < BC + A$. Boiling down \eqref{eqn-normsquares}, we have $\sum_i(\normsq{y_i}+\normsq{v_i}) \sim \delta \nu$, with $\delta < 0$. The condition $\mu = \beta$ additionally implies that the sum of $\normsq{y}+\normsq{v}$, taken over all components of the decomposed quiver, is equal to $\Lambda + \nu$. Hence by the preceding argument we obtain an equation of the form
\begin{equation}
  \delta' \nu = C(\alpha) + \Lambda,\nonumber
\end{equation}
where $\delta' < 0$ and $C(\alpha)$ is a constant depending only on $\alpha$ and the decomposition of the quiver.  For $\Lambda$ sufficiently large, the right side is positive, and so $\nu < 0$. Thus, the critical set cannot be anomalous.
\end{proof}
%-----------------------------------------------------------------------------------------

%=========================================================================================
\subsection{Betti Numbers} \label{sec-betti}
%=========================================================================================

In the preceding section, we established that the Poincar\'e polynomial of $\CG^r_n(\alpha)$, and hence of $\CX^r_n(\alpha)$, does not contain any contributions from points at infinity of the compactification. It remains, however, to enumerate the ordinary, uncompactified critical sets, as well as to compute their Morse indices and equivariant Poincar\'e series. The calculations in this section are standard (see Atiyah-Bott \cite{YangMillsRiemannSurface}, and, especially, Harada-Wilkin \cite{HaradaWilkin}), so we will just sketch the main ideas, leaving the details as an exercise to the reader.

%-----------------------------------------------------------------------------------------
\begin{proposition} \label{prop-ordinary-poincare}
The ordinary critical sets of $\normsq{\mu_T}$ are indexed by the choice of a partition $\lambda$ of $r$, together with disjoint subsets $S_1, \dots, S_{\ell(\lambda)} \subseteq [n]$ labeled by the parts of $\lambda$. For such a critical set $C$, it must be that
\begin{equation}
  P^T_t(C) = (1-t^2)^{-s(\lambda, \rho)} \prod_{j=1}^{\ell(\lambda)} P_t(\CX^{r_j}_{n_j}),\nonumber
\end{equation}
where $r = r_1 + \dots + r_{\ell(\lambda)}$ is the partition, $n_i = \# S_i$, and
\begin{equation}
  s(\lambda, \rho) = \ell(\lambda) + n - 1 - \sum_i n_i.\nonumber
\end{equation}
\end{proposition}
%-----------------------------------------------------------------------------------------
\begin{proof}
This essentially follows by an identical argument to that in Atiyah-Bott \cite[Proposition 7.12]{YangMillsRiemannSurface} and, in the specific context of quiver representations, that of Harada-Wilkin \cite[Proposition 6.10]{HaradaWilkin}. As in the proof of Lemma \ref{main-lemma} the weights $\beta=(\sigma, \rho)$ determine a decomposition of the quiver $Q$ into a quiver $Q_\beta = \sqcup_{j=1}^{\ell(\lambda)} Q_j$, which is a disjoint union of smaller star quivers (see Figure \ref{fig-equivariant-decomposition}). The subset $S_j \subseteq[n]$ is determined by which arrows of $Q$ belong to the component $Q_j$. The representations of the decomposed quiver are, by definition, fixed under the action of $T_\beta$. So choose a complementary sub-torus $K_\beta$ of $T$ so that $T \iso T_\beta \times K_\beta$. The dimension of $T_\beta$ (modulo the diagonal scalars, which act trivially) is precisely $s(\lambda,\rho)$ as given above, and so we have
\begin{equation}
  P^T_t(C) = (1-t^2)^{-s(\lambda, \rho)} P^{K_\beta}_t(C) = (1-t^2)^{-s(\lambda, \rho)} P_t(C / K_\beta),\nonumber
\end{equation}
Furthermore, we have that $C$ is the subspace of $T^\ast \Rep(Q_\beta)$ defined by the vanishing of the residual moment map, and hence
\begin{equation}
  P_t(C / K_\beta) = P_t(T^\ast \Rep(Q_j) \rred G_j ) = \prod_{j=1}^{\ell(\lambda)} P_t( \CX^{r_j}_{n_j}),\nonumber
\end{equation}
\noindent where $G_j$ denotes the automorphisms of the component $Q_j$, i.e. $G_j = P( U(r_j) \times (S^1)^{n_j})$.

\begin{figure}[ht!]
\begin{tikzpicture}

  \node[draw, ellipse, scale=0.5] (Rsink-1) at (0,0) {$r_1$};
  \node[draw, ellipse, scale=0.5] (Rsink-2) at (3,0) {$r_2$};
  \node[draw, ellipse, scale=0.5] (Rsink-3) at (3,-3) {$r_{\ell(\lambda)}$};
  % \dotnodebelow{sink}{c}{$2$};

  \foreach \n in {1, 2, 3}
  {

    \foreach \a in {45, 90, ..., 359}
    {
      \node[draw, ellipse, scale=0.5] (Rv) at ($(Rsink-\n)+(\a:1.00)$) {$1$};
      \path[-stealth] (Rv) edge (Rsink-\n);
      \path[-stealth, dashed] (Rsink-\n) edge[bend left=20] (Rv);
    };

    \node[scale=0.5] (Rd) at ($(Rsink-\n)+(0:1)$) {$\cdots$};
  };

  \node (n) at (0,-3) {$\cdots$};

\end{tikzpicture}
\caption{Decomposition of the quiver corresponding to a critical set.}
\label{fig-equivariant-decomposition}
\end{figure}
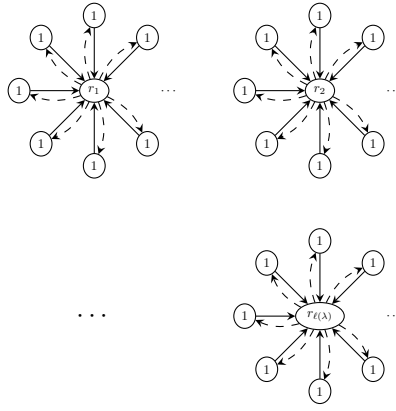

\end{proof}
%-----------------------------------------------------------------------------------------

%-----------------------------------------------------------------------------------------
\begin{proposition} \label{prop-ordinary-index}
Let $(\lambda, S_1, \dots, S_{\ell(\lambda)})$ index an ordinary critical set. Then the complex Morse index of $|\mu_T|^2$ along the critical set, both as a function on $\overline{M}$ as well as on $\partial M$, is equal to
\begin{equation}
  \beta(\lambda, \rho) = r(n-r) + \sum_j r_j(r_j-n_j).\nonumber
\end{equation}
\end{proposition}
%-----------------------------------------------------------------------------------------
\begin{proof} 
The stabilizing torus $T_\beta \subseteq T$ corresponding to $(\lambda, S_1, \dots, S_{\ell(\lambda)})$ takes the form 
\begin{equation}
  T_\beta = \{(s_1, \dots, s_\ell) \times (t_1, \dots, t_n) \in (S^1)^\ell \times (S^1)^n \suchthat t_j = s_i\ \textrm{for}\ j \in S_i \},\nonumber
\end{equation}
and the map $\phi: T_\beta \to U(r)$ determining the twisted action can be taken to be $\phi: s \times t \mapsto \diag(s_1, \dots, s_\ell)$. Recall that the complex Morse index of $\normsq{\mu}$ on a critical set indexed by $\beta$ is equal to the Morse index of $\mu^\beta$, and that the index of the latter is the number of weights, counted with multiplicity, in the isotropy representation of $T_\beta$ which pair negatively with $\beta$. To compute the isotropy representation, we represent the tangent space at $[x,y,z] \in \overline{M}$ by the complex
\begin{equation} \label{diagram-tangent-complex}
   \mathfrak{gl}_r \longrightarrow T^\ast \Rep(Q) \longrightarrow \mathfrak{gl}_r^\ast\nonumber
\end{equation}
\noindent whose first arrow is the differential of the action map and whose second arrow is the differential of the complex $GL(r)$-moment map. Since a fixed point corresponds to a homomorphism $\phi: T \to U(r)$, we can define a new twisted $T$-action on $\Rep(Q)$ by $t \ast x = t \phi^{-1} x$. The new action induces the same $T$-action on $M$. Moreover, with respect to this new action, the above complex is actually a of $T$ representations.  It follows that, in the representation ring of $T$, the isotropy representation is equal to
\begin{equation}
  [\Rep(Q)] + [\Rep(Q)]^\ast - [\mathfrak{gl}_r] - [\mathfrak{gl}_r]^\ast;\nonumber
\end{equation}
\noindent hence, the complex Morse index is simply the total number of non-zero weights in $[\Rep(Q)]-[\gl_r]$, counted with multiplicity. The $T_\beta$-action on $\gl_r$ is induced by the adjoint action, and therefore the number of non-zero weights is $r^2 - \sum_i r_i^2$. The representation on $\Rep(Q)$ is given by
\begin{equation}
  x_{kj} \mapsto s_i t_j^{-1} x_{kj}, k \in S_i,\nonumber
\end{equation}
\noindent from which the number of non-zero weights is given by $rn - \sum_i r_i n_i$. Taking this difference, we find the desired formula for the complex Morse index.
\end{proof}
%-----------------------------------------------------------------------------------------

Now we can prove the main theorem regarding the Betti numbers.

\begin{theorem}\label{main-thm-betti} Let $\alpha$ be generic. Then the Poincar\'e polynomial $P_t(\CX^r_n(\alpha))$ is independent of $\alpha$ and may be computed by the recursion relation
\begin{equation} \label{main-eqn-recursion}
  \frac{P_t(Gr(r,n))}{(1-t^2)^{n-1}} 
    = \sum_{\lambda} \frac{1}{m(\lambda)!} \sum_{\rho \geq \lambda} \frac{t^{2 \beta(\lambda, \rho)}}{(1-t^2)^{s(\lambda, \rho)}}
      {n \choose \rho }
      \prod_{j=1}^{\ell(\lambda)} P_t(\CX^{\lambda_j}_{\rho_j}).\nonumber\end{equation}
\end{theorem}
%-----------------------------------------------------------------------------------------

\begin{proof} The independence of $\alpha$ is a direct consequence of \cite[Corollary 4.2]{Nakajima94}.
By equivariant Morse theory, together with the fact that there are no anomalous critical sets, we have
\begin{align}
  P_t(\overline{M} \red T) &= \frac{P_t^T(\overline{M})}{1-t^2} - \sum_{\text{ordinary}} t^{\lambda_C} P_t^{T \times S^1}(\overline{C}) - \sum_{\text{boundary}} t^{\lambda_C} P_t^{T \times S^1}(C), \nonumber\\
  P_t(\partial M \red T) &= \frac{P_t^T(\partial M)}{1-t^2} - \sum_{\text{ordinary}} t^{\lambda_C} P_t^{T \times S^1}(\partial \overline{C}) - \sum_{\text{boundary}} t^{\lambda_C-2} P_t^{T \times S^1}(C),\nonumber
\end{align}
Using the fact that $P_t(M \red T) = P_t(\overline{M} \red T) - t^2 P_t(\partial M \red T)$, we see that the terms involving boundary critical sets all cancel, and this yields for us
\begin{equation}
    P_t(M \red T) = P_t^T(M) - \sum_{\text{ordinary}} t^{\lambda_C} \left(P_t^{T \times S^1}(\overline{C}) - t^2 P_t^{T \times S^1}(\partial C) \right). \nonumber
\end{equation}
Furthermore, since $P_t^T(\overline{C}) - t^2 P_t^T(\partial C) = P_t^T(C)$, we obtain 
\begin{equation}
    P_t(M \red T) = P_t^T(M) - \sum_{\text{ordinary}} t^{\lambda_C} P_t^{T \times S^1}(C). \nonumber
\end{equation}
Using this, together with propositions \ref{prop-ordinary-index} and \ref{prop-ordinary-poincare} above, we obtain a recursion relation as a sum over partitions $\lambda$ and
disjoint subsets $S_1, \dots, S_{\ell(\lambda)}$. However, the contribution from each critical set depends only on the partition $\lambda$ and the tuples $\rho = (n_1, \dots, n_{\ell(\lambda)})$ of sizes of the subsets. The number of such subsets is exactly
\begin{equation}
  {n \choose \rho} := {n \choose n_1} {n-n_1 \choose n_2} \cdots {n-n_1-\dots-n_{\ell{\lambda}-1} \choose n_{\ell(\lambda)}}.
\end{equation}

Furthermore, the sum over \emph{all} partitions and subsets $S_1, \dots, S_{\ell(\lambda)}$ necessarily overcounts the critical sets. The reason is that if a part $r_i$ has some multiplicity, the corresponding subsets of $[n]$ are indistinguishable, and hence we should divide by a factor of $m(r_i)!$, where $m(r_i)$ denotes the multiplicity of the part. To account for this, we divide by an overall factor of $m(\lambda)!$, where 
\begin{equation}
  m(\lambda)! = \prod_{r_i\ \textrm{distinct}} r_i!\nonumber
\end{equation}
and thus we obtain the desired recursion relation.
\end{proof}
%-----------------------------------------------------------------------------------------

We illustrate the recursion relation in the following special case.
%-----------------------------------------------------------------------------------------
\begin{corollary} In the rank 2 case, we have
\begin{align}
  \begin{split}
  P_t(\CX^2_n) &= 
  \frac{1}{(1-t^2)^{n-1}} P_t(Gr(2,n))
     - \sum_{k=3}^{n-1} {n \choose k} \frac{t^{4(n-k)}}{(1-t^2)^{n-k}} P_t(\CX^2_k) \nonumber\\
     &- \frac{1}{2}\sum_{k_1=1}^{n-1} \sum_{k_2=1}^{n-k_1} 
       {n \choose k_1} {n-k_1 \choose k_2} \frac{t^{4n-4-2k_1-2k_2}}{(1-t^2)^{n+1-k_1-k_2}}.\nonumber
       \end{split}
\end{align}
\end{corollary}

%-----------------------------------------------------------------------------------------
\begin{remark} 
Konno calculated the rank 2 Poincar\'e polynomials using Morse theory with the $S^1$-action \cite{KonnoPolygon}. This method was generalized to higher ranks in the first author's PhD thesis, but for $r \geq 3$ the calculation is significantly more complicated. In particular, one obtains a complicated sum of terms that depend sensitively on the parameter $\alpha$, despite the fact that the sum itself is independent of $\alpha$. On the other hand, the recursion relation of Theorem \ref{main-thm-betti} may be used to compute the Poincar\'e polynomials for large rank and twist very quickly. 
\end{remark}

%-----------------------------------------------------------------------------------------
\begin{figure}[ht]
\begin{tabular}{ccc}
\includegraphics[width=6cm]{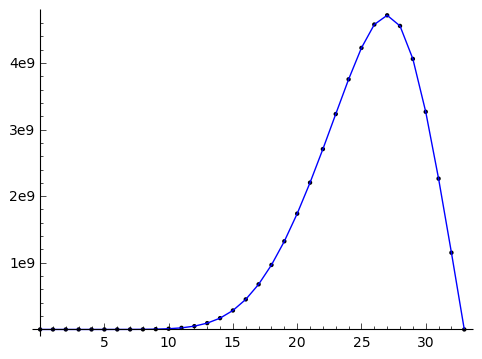} &
\includegraphics[width=6cm]{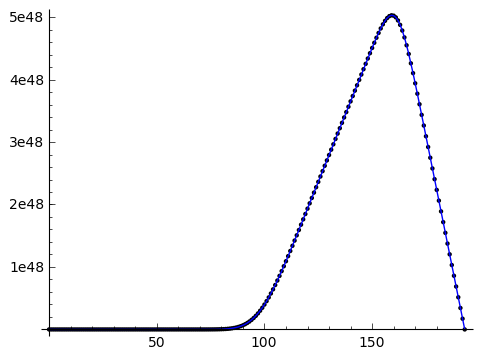} 
\end{tabular}
\caption{Typical plots of Betti numbers. Left: $r=3,\;n=20$. Right: $r=3,\;n=100$.  Data points have been joined to demonstrate their smooth appearance.}
\label{fig-betti-plot}
\end{figure}
%-----------------------------------------------------------------------------------------

%-----------------------------------------------------------------------------------------
\begin{remark} Hausel and Rodriguez-Villegas observed in \cite{HRV2013} the smooth appearance of plots of Betti numbers of some high-dimensional quiver varieties and other non-compact hyperk\"ahler manifolds, and proved that in several instances these converge (in an appropriate sense) to well-known probability distributions. This phenomenon seems also to occur in the case of hyperpolygon spaces, as exhibited in Figure \ref{fig-betti-plot}.
\end{remark}
%-----------------------------------------------------------------------------------------

%=========================================================================================
% "Hyperpolygons and Hitchin Systems"
%
% Jonathan Fisher
% jonathan.fisher@uni-hamburg.de
%
% Steven Rayan
% stever@math.toronto.edu
% 
% TeX source for Section 4: "Hitchin Systems"
%
%=========================================================================================

%==============================================================================
\section{Hitchin Systems}
%===============================================================================

%==============================================================================
\subsection{Parabolic Higgs Bundles}
%===============================================================================

In this section, we will show that $\CX^r_n(\alpha)$ can be naturally viewed as a moduli space of parabolic Higgs bundles.\footnote{For general background, see \cite{Simpson90, AHH, BodenYokogawa, LogaresMartens}, etc.} Let $X$ be a compact connected Riemann surface and $E \to X$ a rank $r$ holomorphic vector bundle on it.  Use $K_X$ to denote the canonical line bundle of $X$.  Let $D = \sum_{i=1}^n p_i$ be a fixed divisor on $X$, and assume that no two $p_i$ are the same. 

%-----------------------------------------------------------------------------------------
\begin{definition}
A \emph{minimal parabolic structure} on $E$ is a choice at each $p_i$ of a partial flag $0 \subset L_i \subset E_{p_i}$, where $L_i$ is a line. A \emph{minimal parabolic Higgs field} is an $\CO_X$-linear map $\phi \in H^0(X, \Endo (E) \otimes K_X(D) )$, such that at each point $p_i$, the residue of $\phi_i$ of $\phi$ is strictly triangular with respect to the flag (i.e. $\phi(E_{p_i}) \subseteq L_i, \phi_i(L) = 0$).
\end{definition}

%-----------------------------------------------------------------------------------------

\noindent Note that if $\phi$ is strictly parabolic, then the  $\phi_i$ are nilpotent of order $2$, and have rank at most $1$, and we have that they lie in the closure of the minimal nilpotent orbit of $\mathfrak{sl}_r$.

Now consider $[x,y] \in \CX^r_n(\alpha)$ represented by $(x,y)$. We can define a Higgs field on $\BP^1$ by setting 
\begin{equation} \label{eqn-integrable-higgs-defn}
  \phi(z) = \sum_{i=1}^n \frac{\phi_i dz}{z-p_i},\nonumber
\end{equation}
where $z$ is an affine coordinate on $\BP^1$ and where $\phi_i = x_i y_i$. The complex moment map condition $\sum_i x_i y_i =0$ ensures that $\phi$ is well-defined as a section of $\Endo(E) \otimes K(D)$, and moreover the complex moment map conditions $y_i x_i =0$ imply that $\phi_i$ is strictly triangular with respect to the flag $0 \subset \mathrm{span} \{x_i\} \subset \BC^r$. Different choices of representative $(x,y)$ differ by $g \in GL(r)$, and this modifies $\phi$ by $\phi \mapsto g \phi g^{-1}$. Hence the equivalence class $(E, \phi)$, where $E = \BC^r \times \BP^1$ with parabolic structure defined by the $x_i$, is well-defined independent of the choice of representative $(x,y)$. Moreover, by Proposition \ref{prop-nilpotent-resolution}, we see that the pair $(E, \phi)$ determines $[x,y] \in \CX^r_n(\alpha)$ uniquely. Summarizing, we have the following.

%-----------------------------------------------------------------------------------------

\begin{theorem}\label{thm-integrable-parabolic}
Let $D = \sum_{i=1}^n p_i$ be a divisor of $n$ distinct points on $\BP^1$.  Then the hyperpolygon space $\CX_n^r(\alpha)$ is naturally identified with a moduli space of parabolic Higgs bundles on $\BP^1$ (with respect to the divisor $D$), whose underlying bundle type is trivial,
and such that the residue of the Higgs field at each point in the divisor lies in the closure of the minimal nilpotent orbit of $\mathfrak{sl}_r$.
\end{theorem}

%-----------------------------------------------------------------------------------------
\begin{remark} In \cite{GodinhoMandini}, it was shown when $r=2$ that
the stability condition on quiver representations is equivalent to parabolic slope stability,
for an appropriate choice of parabolic weights. We suspect that this remains true in the
higher rank case, but as we shall not need this result in subsequent arguments, we make no attempt to prove it here.
\end{remark}
%-----------------------------------------------------------------------------------------

Let $K \iso \CO(-2)$ stand for the canonical line bundle on $\BP^1$, and let $E = \BC^r \times \BP^1$ be the trivial rank $r$ vector bundle. We define the \emph{Hitchin base} to be the affine space
\begin{equation}
  B = \bigoplus_{i=2}^{r} H^0(\BP^1, K^i(D)),\nonumber
\end{equation}
\noindent in which $D = \sum_{i=1}^n p_i$ is a divisor of distinct points on $\BP^1$, fixed once and for all. Let $(x,y)$ represent an equivalence class $[x,y] \in \CX^r_n(\alpha)$, and let $\phi$ be the associated Higgs field as above. We define the \emph{Hitchin map} $\mathbf h: \CX^r_n(\alpha) \to B$ to be 
\begin{equation}
  \mathbf h: [x,y] \mapsto (\Tr(\phi^2), \Tr(\phi^3), \dots, \Tr(\phi^r)).\nonumber
\end{equation}

%-----------------------------------------------------------------------------------------
\begin{lemma} The Hitchin map is well-defined.
\end{lemma}
%-----------------------------------------------------------------------------------------
\begin{proof} Since the trace of $\phi^k$ depends only on the conjugacy class, $\mathbf h$ is indepdendent of the choice of representative of $[x,y]$. It remains to show that $\Tr(\phi^k)$, which is naturally a section of $K(D)^k$, has only first order poles along $D$ so that it may be naturally identified with a section of $K^k(D)$. Note that $\Tr(\phi^k)$ is given by the expression
\begin{equation}
  \sum_{i_1, \dots, i_k} \frac{\Tr(\phi_{i_1} \dots \phi_{i_k})}{(z-p_{i_1})\dots(y-p_{i_k})} (dz)^{\otimes k}.\nonumber
\end{equation}
If we have $i_j = i_\ell$ for any $j,\ell$, then the cyclic property of the trace map together with the nilpotency condition on the residues imply that the corresponding term in the sum vanishes.  Hence, the above sum contains only first order poles along $D$.

\end{proof}
%-----------------------------------------------------------------------------------------

%-----------------------------------------------------------------------------------------
\begin{proposition} We have $\dim B = \frac{1}{2} \dim \CX^r_n(\alpha)$.
\end{proposition}
%-----------------------------------------------------------------------------------------
\begin{proof} The line bundle $K^i(D)$ has degree $n-2i$,
and hence $h^0(K^i(D)) = n-2i+1$.
Summing these from $i=2$ to $i=r$, we obtain the result.
\end{proof}
%-----------------------------------------------------------------------------------------

Next we will show that the map $\mathbf h$ defines a coisotropic fibration.
Define the components $I_m$ of $\mathbf h$ via
\begin{equation}
  I_m (dz)^m = \Tr((\phi(z)dz)^m) \in H^0(\BP^1, K^m(D)).\nonumber
\end{equation}
By construction, the $I_m$ are rational functions in $z$, and hence we may expand them
as formal power series $I_m(z) = \sum_n I_m^n z^n$. (By a change of coordinate if necessary,
we can assume that none of the $p_i$ is zero, so that $\phi(z)$ is regular at $0$.)
It will be convenient to treat $z$ as a formal parameter, so that $I_m(z)$ can be thought
of as a generating function\footnote{We think of the coefficients $I_m^i$
as coordinatizing the Hitchin base $B$. Since they are the Taylor coefficients of 
a rational function, only finitely many of them are algebraically independent.} for the coefficients $I_m^n$.
Trivially, we have:
%-----------------------------------------------------------------------------------------
\begin{lemma} We have $\{I_m^i, I_n^j\}=0$ for all $i,j$ if and only if 
$\{I_m(z), I_n(w)\} = 0$ as an element of $\BC[[z,w]]$, where the formal parameters
$z,w$ are Casimirs of the Poisson bracket.
\end{lemma}

Next, we define a matrix-valued formal power series as follows:
\begin{equation}
  \Delta(z,w) := \frac{\phi(z)}{w-z} + \frac{\phi(w)}{z-w}.\nonumber
\end{equation}
This matrix encodes the Poisson structure, in the following sense:
%-----------------------------------------------------------------------------------------
\begin{lemma} We have
  \[ \{ \phi_{ij}(z), \phi_{kl}(w) \} = \delta_{jk} \Delta_{il}(z,w) - \delta_{il} \Delta_{kj}(z,w). \]
\end{lemma}
%-----------------------------------------------------------------------------------------
\begin{proof} The result follows from index calculus.  In symbols,
\begin{align}
  \{ \phi_{ij}(z), \phi_{kl}(w) \} &= \sum_{m,n} \frac{\{(\phi_m)_{ij}, (\phi_n)_{kl} \}}{(z-p_m)(w-p_n)} \nonumber\\
   &= \sum_{m,n} \frac{\delta_{mn} \delta_{jk} (\phi_m)_{il}-\delta_{mn} \delta_{il} (\phi_m)_{kj}}{(z-p_m)(w-p_n)} \nonumber\\
   &= \sum_m \frac{\delta_{jk} (\phi_m)_{il}-\delta_{il} (\phi_m)_{kj}}{(z-p_m)(w-p_m)} \nonumber\\
   &= \delta_{jk} \Delta_{il}(z,w) - \delta_{il} \Delta_{kj}(z,w),\nonumber
\end{align}
as claimed.
\end{proof}
%-----------------------------------------------------------------------------------------

%-----------------------------------------------------------------------------------------
\begin{proposition} \label{prop-coisotropic}
The components of the Hitchin map $\mathbf h: \CX^r_n(\alpha) \to B$  Poisson commute.
\end{proposition}
%-----------------------------------------------------------------------------------------
\begin{proof} First, note that 
\begin{equation}
  \Tr(\phi(z)^m) = \sum_{i_1, \dots, i_m} \phi_{i_1 i_2}(z) \phi_{i_2 i_3}(z) \cdots \phi_{i_m i_1}(z).\nonumber
\end{equation}
After taking the Poisson bracket $\{I_m(z), I_n(w)\}$, expanding, and applying the Leibniz rule,
we have
\begin{align}
 \begin{split}
  \{I_m(z), I_n(w)\} &= \sum_{\stackrel{a,i_1,\dots, i_m}{b,j_1, \dots, j_n}} 
     \phi_{i_1 i_2}(z) \cdots \widehat{\phi_{i_a i_{a+1}} (z)} \cdots \phi_{i_m i_1}(z)\nonumber \\
     &\ \ \ \  \times \phi_{j_1 j_2}(w) \cdots \widehat{\phi_{j_b j_{b+1}} (w)} \cdots \phi_{j_m j_1}(w)\nonumber \\
      &\ \ \ \  \times \left( \delta_{i_{a+1} j_b} \Delta_{i_a j_{b+1}}(z,w) - \delta_{i_a j_{b+1}} \Delta_{j_b i_{a+1}}(z,w) \right)\nonumber
    \end{split} \\
   \begin{split}
      &= \sum_{a,b} \Tr\left(\phi^{m-1}(z) \Delta(z,w) \phi^{n-1}(w) \right) \nonumber\\
       &\ \ - \sum_{a,b} \Tr\left( \phi^{n-1}(w) \Delta(z,w) \phi^{m-1}(z)\right) \nonumber\end{split} \\
      &= mn \Tr\left(\Delta(z,w) [\phi(w)^{n-1}, \phi(z)^{m-1}] \right).\nonumber
\end{align}
Now recall the definition of $\Delta(z,w)$ as  $\phi(z)/(w-z) + \phi(w)/(z-w)$.
This is a sum of two terms, one commuting with $\phi(z)$ and the other commuting with $\phi(w)$, 
hence the above trace vanishes identically.
\end{proof}
%-----------------------------------------------------------------------------------------

%=========================================================================================
\subsection{Spectral Curves}
%=========================================================================================
In the preceding section, we established that the Hitchin map $\mathbf h: \CX^r_n(\alpha) \to B$
endows $\CX^r_n(\alpha)$ with exactly $\frac{1}{2} \dim \CX^r_n(\alpha)$ commuting Hamiltonians.
However, to establish complete integrability it remains to show that these Hamiltonians
are functionally independent. To do this, we need to study the associated spectral 
curves, which are well-known objects in the world of Higgs bundle moduli spaces \cite{HitchinStableBundles, DonagiMarkman, Donagi}.  Indeed, a spectral curve arising from the Hitchin map is a moduli space of Higgs bundles with fixed characteristic data.

To be more precise in the parabolic case, denote by $\mbox{Tot}(K(D))$ the complex surface given by the total space of $K(D)$, and let $\pi:\mbox{Tot}(K(D)) \to \BP^1$ be the natural map.  Then, $\pi^\ast K(D)$ has a canonical section $\lambda \in H^0(K(D), p^\ast K(D))$ which
in local coordinates is given by
\begin{equation}
  (z, \lambda) \mapsto \lambda.\nonumber
\end{equation}

%-----------------------------------------------------------------------------------------
\begin{definition} For $b \in B$ we define the \emph{spectral curve} $\Sigma_b$ to be
the one-dimensional subvariety of $\mbox{Tot}(K(D))$ given by the zero locus of
\begin{equation}
  \lambda^r + b_1(z) \lambda^{r-1} + \dots + b_r(z).\nonumber
\end{equation}
We denote by $\nu_b: \tilde{\Sigma}_b \to \Sigma_b$ the normalization of $\Sigma_b$.
\end{definition}
%-----------------------------------------------------------------------------------------

%-----------------------------------------------------------------------------------------
For any $z\in\BP^1$, $\#((\pi|_{\Sigma_b})^{-1}(z))\leq r$, and over a generic $z\in\BP^1$, $\#((\pi|_{\Sigma_b})^{-1}(z))=r$.  Therefore, $\Sigma_b$ is a generically $r:1$ cover of the projective line.  The same is true for $\tilde{\Sigma}_b$.  We say that $\Sigma$ and $\Sigma_b$ have \emph{degree} $r$.  The curve $\Sigma_b$ is called ``spectral'' because if $\phi$ is such that $\phi\in\mathbf h^{-1}(b)$, then $(\pi|_{\Sigma_b})^{-1}(z)$ is the set of eigenvalues of $\phi(z)$.  The number $\#((\pi|_{\Sigma_b})^{-1}(z))$ is less than $r$ precisely where $\phi$ has repeated eigenvalues.  

A key ingredient in the proof of the integrability theorem, Theorem \ref{main-thm-integrability} below, is the ``spectral correspondence''.  This correspondence associates to a Higgs bundle $(E,\phi)$ with $\phi\in\mathbf h^{-1}(b)$ a spectral pair $(\Sigma_b,L)$, where $L$ is a line bundle on $\Sigma_b$ for which $(\pi|_{\Sigma_b})_*L=E$.  In other words, it is an isomorphism between a space of Higgs bundles on one curve and the space of line bundles of a fixed degree on another curve, usually of higher genus.  (Loosely speaking, if $\Sigma_b$ consists of the eigenvalues of $\phi$, then $L$ consists of its eigenspaces.) This correspondence holds whether $(E,\phi)$ is an ordinary Higgs bundle or a parabolic one, and whether the pair $(E,\phi)$ is stable or not --- although stability of $(E,\phi)$ has implications for the geometry of $\Sigma_b$ (for example, $\Sigma_b$ for a stable $(E,\phi)$ cannot have more than one irreducible component).  The correspondence is constructed rigorously in \cite{HitchinStableBundles,BNB89}.  For the parabolic case, the spectral correspondence is explicitly discussed in \cite{LogaresMartens}.

%-----------------------------------------------------------------------------------------

The following lemma characterizes the behaviour of a spectral curve over the marked points in $\BP^1$.

%-----------------------------------------------------------------------------------------
\begin{lemma} \label{lemma-spectral-singularities}
Let $b \in B$ be generic. Then $\Sigma_b$ is smooth away from
$\pi^{-1}(D)$. Near any singular point $p_j \in D$, the equation for $\Sigma_b$ is of the form
\begin{align}
  \begin{split}
    &\lambda^r + \sum_{i=2}^r z^{\floor{(i+1)/2}} a_i(z) \lambda^{r-i} \nonumber
  \end{split}
\end{align}
where each $a_i(z)$ is a polynomial in $z$ and $\lambda$ is a vertical coordinate in the line bundle $K_X(D)$.
\end{lemma}
%-----------------------------------------------------------------------------------------
\begin{proof} Smoothness away from $D$ follows by a Bertini argument.
The nilpotency condition implies that $\Tr(\phi^i)$ vanishes to order $i-1$
along $D$. Since the coefficients of the characteristic polynomial are polynomials in
the $\Tr(\phi^i)$, homogeneous in $\phi$, we just have to show that any term of the 
form $\Tr(\phi^{i_1}) \dots \Tr(\phi^{i_\ell})$, with $i_1 + \dots + i_\ell = i$, vanishes
to order at least $\floor{(i+1)/2}$. By the nilpotency condition, such a term vanishes
to order
\begin{align}
  (i_1-1) + \dots + (i_\ell-1) = i_1 + \dots + i_\ell - \ell = i - \ell,\nonumber
\end{align}
and hence is smallest when $\ell$ is as large as possible. Since $\Tr(\phi)$ is identically
zero, each $i_k$ must be at least $2$. There are two cases, depending on whether $i$ is
even or odd.

If $i$ is even, the longest partition
is $2 + \dots + 2$, so $\ell \leq i/2$ and the coefficient of $\lambda^{r-i}$ vanishes
to order $i/2 = \floor{(i+1)/2}$.

If $i$ is odd, then the longest partition is $3 + 2 + \dots + 2$.
Hence $\ell \leq 1 + (i-3)/2 = (i-1)/2$, and the coefficient of $\lambda^{r-i}$ vanishes
to order $1-(i-1)/2 = (i+1)/2 = \floor{(i+1)/2}$.
\end{proof}
%-----------------------------------------------------------------------------------------

\subsection{Integrability via Normalization}

%-----------------------------------------------------------------------------------------
It is worth pointing out that the moduli space of strictly parabolic Higgs bundles is embedded as a symplectic leaf
in the moduli space of general parabolic Higgs bundles \cite{LogaresMartens}, and it is known to
be completely integrable. However, for $r \geq 3$ the hyperpolygon space maps to a subvariety of positive codimension, and it is not obvious \emph{a priori} that the restriction of the Hitchin map to this subvariety has the expected number of functionally independent components.

%%  However, we take the viewpoint that representation spaces of star quivers should
%% be regarded as ``toy models'' of Hitchin systems, and so we will instead prove integrability
%% by direct investigation.  

To proceed, we examine local rings of functions on normalizations of spectral curves associated to parabolic Higgs bundles coming from the quiver construction. To simplify notation, we replace $\phi \in H^0(\BP^1, \Endo(E) \otimes K(D))$ by
\begin{equation}
  \phi(z) \prod_{i=1}^n (z-p_i) \in H^0(\BP^1, \Endo(E) \otimes \CO(n-2))\nonumber
\end{equation}
and modify the Hitchin map and spectral curve accordingly. This allows us to view $\phi$ as a polynomial-valued matrix, subject to the conditions $\phi^2=0$ and $\rk(\phi) \leq 1$ at each of the marked points $p_i \in D$.

%-----------------------------------------------------------------------------------------

%-----------------------------------------------------------------------------------------
\begin{proposition} \label{prop-normalization}
Let $\nu: \tilde{\Sigma} \to \Sigma$ be the normalization of a degree $r$ spectral curve $\Sigma$.
For any of the marked points $p_i$, there is an neighbourhood $U$ of $\nu^{-1}(p_i)$ on which
$\lambda^2 / (z-p_i)$ is a regular function on $U$.
\end{proposition}
%-----------------------------------------------------------------------------------------
\begin{proof}  Pick an open affine chart $U$ with coordinates $(z, \lambda)$ near $p_j$ on 
which the spectral curve is given by an equation of the form 
\begin{equation}
  f(z, \lambda) = \lambda^r + \sum_{i=2}^r z^{\floor{(i+1)/2}} a_i(z) \lambda^{r-i} = 0
\end{equation}
as guaranteed by Lemma \ref{lemma-spectral-singularities}. By the definition of normalization, 
the regular functions on $\nu^{-1}(U)$ are rational functions in $(z, \lambda)$
satisfying a monic polynomial with coefficients in $\BC[z, \lambda] / \ev{f}$.
We take two cases: $r$ even and $r$ odd.

In the first case, $r = 2\ell$. Dividing $f$ by $z^\ell$, we find the equation
\begin{equation}
  \left(\frac{\lambda^2}{z}\right)^\ell + \dots  = 0,\nonumber
\end{equation}
\noindent which shows $\lambda^2/z$ satisfies a monic polynomial with regular coefficients. Thus
$\lambda^2 / z$ is (locally) regular on $\tilde{\Sigma}$.

In the second case, $r = 2\ell-1$. Multipliy $f$ by $\lambda$ and divide by $z^\ell$.
This gives us an equation
\begin{equation}
  \left(\frac{\lambda^2}{z}\right)^\ell + \dots = 0.\nonumber
\end{equation}
It is an easy exercise, left to the reader, to verify using Lemma \ref{lemma-spectral-singularities} that the terms represented by $\cdots$ are polynomial in $\lambda^2/z$, of degree less than $\ell$, with regular coefficients. Hence $\lambda^2/z$ defines a regular function on the normalization.
\end{proof}
%-----------------------------------------------------------------------------------------

We are finally equipped to establish the main theorem, regarding complete integrability.  We state the theorem again here, and then prove it.

\begin{theorem}\label{main-thm-integrability} For $r=2,3$ the Hitchin map $\mathbf h: \CX^r_n(\alpha)$ is surjective, and hence its generic fibres are Lagrangian subvarieties.
\end{theorem}
%-----------------------------------------------------------------------------------------
\begin{proof}

By the spectral correspondence, there is a 1-1 correspondence between Higgs fields
(independent of stability and bundle type) with characteristic polynomial $\mathbf h(\phi) \in B$ and line bundles in the complement of a divisor in the Jacobian\footnote{The Jacobian contains a divisor whose line bundles have nontrivial direct image on $\BP^1$.} of the normalized spectral curve, $\tilde{\Sigma}$.  By Proposition \ref{prop-normalization},
we have that $\lambda^2 = 0$ for any $(z, \lambda) \in p^{-1}(D)$. Hence for any line bundle $L$
on $\tilde{\Sigma}$, the induced Higgs field $\phi$ satisfies $\phi^2 = 0$ along $D$, i.e.
the nilpotency condition is satisfied on the entire Hitchin fibre. For ranks $r=2,3$, the nilpotency condition $\phi(p_i)^2=0$ implies $\rk(\phi(p_i))\leq 1$, and hence $\phi(p_i)$ lies in the closure of the minimal nilpotent orbit. Hence, $\mathbf h: \CX^r_n(\alpha) \to B$ is surjective, and the components of $\mathbf h$ are functionally independent.
\end{proof}
%-----------------------------------------------------------------------------------------

%-----------------------------------------------------------------------------------------
\begin{example} Consider a degree 3 the spectral curve $\Sigma$, cut out of $\mbox{Tot}(K(D))$ near a marked point $p_j$ by
\begin{equation}
  f(z, \lambda) = \lambda^3 - z\lambda - z^2 = 0.\nonumber
\end{equation}
This specral curve is pictured in Figure \ref{fig-rank-3-spectral-curve}. Multiplying by $\lambda^{-2}$, we get $\lambda - (z/ \lambda) - (z/ \lambda)^2 = 0$, and 
hence $q = z/\lambda$ is a regular function on the corresponding open set in the normalization. It is easy to check
that setting $z = q\lambda$ defines the normalization $\tilde{\Sigma}$ of $\Sigma$. 
Since $\tilde{\Sigma}$ is smooth, a torsion free sheaf is locally free, and hence since it is affine 
any line bundle on $\tilde{\Sigma}$ corresponds to $\BC[\tilde{\Sigma}]$ as a
$\BC[\tilde{\Sigma}]$-module. As a $\BC[\Sigma]$-module, $\BC[\tilde{\Sigma}]$ is generated
by the elements $\{1, q, q^2\}$, with $\lambda = q+q^2$ and $z = \lambda q = q^2 + q^3$.
The Higgs field $\phi$ is given by the action of $\lambda$ on this module, and so the spectral data is equivalent to
\begin{align}
  \phi(1) &= \lambda \cdot 1 = q+q^2 \nonumber\\
  \phi(q) &= \lambda \cdot q = q^2 + q^3 = z\nonumber \\
  \phi(q^2) &= \lambda \cdot q^2 = q^3 + q^4 = (q^2+q^3)q = zq.\nonumber
\end{align}
Hence, a local representative for the Higgs field near $p_j$ is
\begin{equation}
  \phi(z) = \begin{bmatrix} 0 & z & 0 \\ 1 & 0 & z \\ 1 & 0 & 0 \end{bmatrix},\nonumber
\end{equation}
which is a rank 1 matrix, nilpotent of order $2$ at $p_j$ (i.e. at $z=0$).
\begin{figure}[h]
\includegraphics[width=6cm]{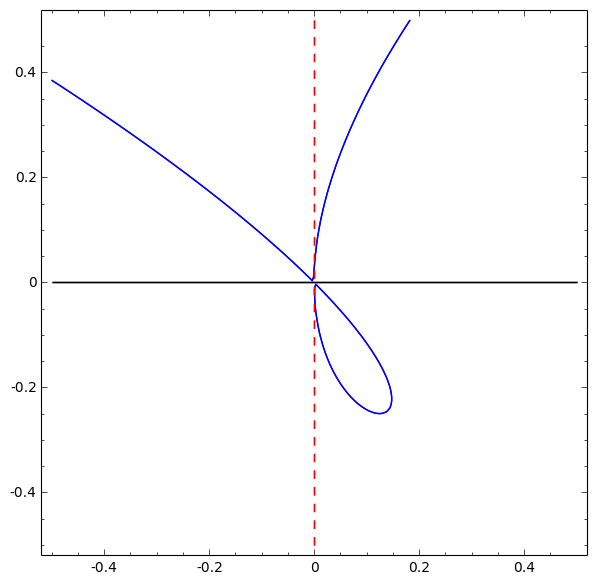}
\caption{Typical singularity in a rank 3 spectral curve.}
\label{fig-rank-3-spectral-curve}
\end{figure}
\end{example}
%-----------------------------------------------------------------------------------------

%-----------------------------------------------------------------------------------------
\begin{remark} The Hitchin map $\mathbf h: \CX^4_n(\alpha) \to B$ is not surjective.
Consider the degree $4$ the spectral curve $\Sigma$ given locally by
\begin{equation}
  f(z, \lambda) = \lambda^4 - z a(z) \lambda^2  - z^2 b(z) \lambda - z^2 c(z) = 0.\nonumber
\end{equation}
This spectral curve is pictured in Figure \ref{fig-rank-4-spectral-curve}. Setting $q = \lambda^2 / z$, we find that
\begin{equation}
  q^2 - a(z) q - b(z) \lambda - c(z) = 0,\nonumber
\end{equation}
and for generic $a,b,c$ this defines the normalization $\tilde{\Sigma}$ of $\Sigma$. 
Again, any line bundle on $\tilde{\Sigma}$ corresponds to $\BC[\tilde{\Sigma}]$ as a
$\BC[\tilde{\Sigma}]$-module. As a $\BC[\Sigma]$-module, $\BC[\tilde{\Sigma}]$ is generated
by the elements $\{1, \lambda, q, \lambda q\}$, with $\lambda^2 = zq$.
The Higgs field $\phi$ is given by the action of $\lambda$ on this module, so we have
\begin{align}
  \phi(1) &= \lambda \cdot 1 = \lambda\nonumber \\
  \phi(\lambda) &= \lambda \cdot \lambda = \lambda^2 = zq \nonumber\\
  \phi(q) &= \lambda \cdot q = \lambda q \nonumber\\
  \phi(\lambda q) &= \lambda^2 q = z q^2 = z(a q + b\lambda + c .)\nonumber
\end{align}
Hence,
\begin{equation}
  \phi(z) = \begin{bmatrix} 0 & 0 & 0 & zc \\ 1 & 0 & 0 & zb \\ 0 & z & 0 & za \\ 0 & 0 & 1 & 0 \end{bmatrix},\nonumber
\end{equation}
which satisfies $\phi^2 = 0$ at $z=0$. However, we see that $\phi$ has rank 2 at $z=0$, and hence is not contained in the closure of the minimal nilpotent orbit.
\begin{figure}[h]
\includegraphics[width=6cm]{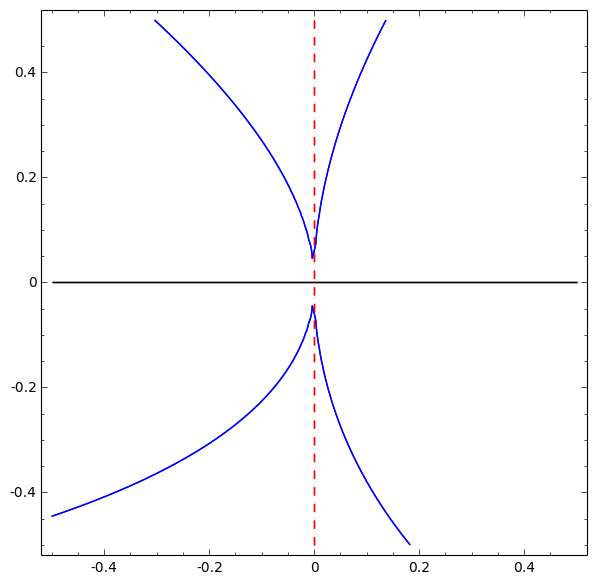}
\caption{Typical singularity in a rank 4 spectral curve.}
\label{fig-rank-4-spectral-curve}
\end{figure}
\end{remark}
%-----------------------------------------------------------------------------------------

%===============================================================================
% \newpage
\bibliographystyle{plain}
\bibliography{polygon}
%===============================================================================

%===============================================================================
\end{document}